\documentclass[11pt,twoside,english]{article}
\usepackage{amsmath}
\usepackage{amsfonts}
\usepackage{mathrsfs}
\usepackage{color}
\usepackage{ amsmath, amsfonts, amssymb, amsthm, amscd}
\usepackage[T1]{fontenc}
\usepackage[english]{babel}
\usepackage[noinfoline]{imsart}

\newtheorem{theorem}{Theorem}[section]
\newtheorem{lemma}[theorem]{Lemma}
\newtheorem{proposition}[theorem]{Proposition}

\newtheorem{example}{Example}[section]
\newtheorem{definition}[theorem]{Definition}

\newtheorem{remark}{Remark}[section]
\newcommand{\cE}{\ensuremath{\mathcal E}}

\newcommand{\cH}{\ensuremath{\mathcal H}}

\newcommand{\cO}{\ensuremath{\mathcal O}}

\newcommand{\cR}{\ensuremath{\mathcal R}}

\newcommand{\cW}{\ensuremath{\mathcal W}}

\newcommand{\bbN}{{\ensuremath{\mathbb N}} }

\newcommand{\bbR}{{\ensuremath{\mathbb R}} }

\newcommand{\N}{\mathbb{N}}

\newcommand{\be}{\begin{equation}}
\newcommand{\ee}{\end{equation}}
\newcommand{\beq}{\begin{eqnarray}}
\newcommand{\eeq}{\end{eqnarray}}

\newcommand{\1}{{1} \hspace{-0.25 em}{\rm I}}

\newcommand{\Ne}{\N^{\ast}}

\newcommand{\HH}{{\cal H}}

\newcommand{\OO}{{\cal O}}

\newcommand{\R}{\mathbb{R}}

\newcommand{\ced}{\end{proof}}

\newcommand{\vphi}{\varphi}

\begin{document}
\begin{frontmatter}
\title{Maximum Principle for Quasilinear Stochastic PDEs with Obstacle  }
\date{}
\runtitle{}
\author{\fnms{Laurent}
 \snm{DENIS}\corref{}\ead[label=e1]{ldenis@univ-evry.fr}}
\thankstext{T1}{The work of the first and third author is supported by the chair \textit{risque de cr\'edit}, F\'ed\'eration bancaire Fran\c{c}aise}
\address{Universit\'e
d'Evry-Val-d'Essonne-FRANCE
\\\printead{e1}}
\author{\fnms{Anis}
 \snm{MATOUSSI}\corref{}\ead[label=e2]{anis.matoussi@univ-lemans.fr}}
\thankstext{t2}{The research of the second author was partially supported by the Chair {\it Financial Risks} of the {\it Risk Foundation} sponsored by Soci\'et\'e G\'en\'erale, the Chair {\it Derivatives of the Future} sponsored by the {F\'ed\'eration Bancaire Fran\c{c}aise}, and the Chair {\it Finance and Sustainable Development} sponsored by EDF and Calyon }
\address{
 LUNAM Université, Université du Maine - FRANCE \\\printead{e2}}
% \affiliation{Some University}

\author{\fnms{Jing}
 \snm{ZHANG}\corref{}\ead[label=e3]{jzhang@univ-evry.fr}}
%\thankstext{T1}{The work of the first author is supported by the chair \textit{risque de cr\'edit}, F\'ed\'eration bancaire Fran\c{c}aise}
\address{Universit\'e
d'Evry-Val-d'Essonne -FRANCE\\\printead{e3}}

\runauthor{L. Denis, A. Matoussi and J. Zhang}

\begin{abstract}
We prove a maximum principle for local solutions of quasilinear
stochastic PDEs with obstacle (in short OSPDE). The proofs are based
on a version of It\^o's formula and estimates for the positive part
of a local solution which is non-positive on the lateral boundary.

\end{abstract}


\end{frontmatter}

\section{Introduction}
In this paper, we consider an obstacle problem for the following
parabolic Stochastic PDE (SPDE in short)
\begin{equation}\label{SPDEO}\left\{ \begin{split}&du_t(x)=\partial_i  \left(a_{i,j}(x)\partial_ju_t(x)+g_i(t,x,u_t(x),\nabla
u_t(x))\right)dt+f(t,x,u_t(x),\nabla u_t(x))dt \\&\quad \quad\ \  \
\ \ \ +\sum_{j=1}^{+\infty}h_j(t,x,u_t(x),\nabla
u_t(x))dB^j_t +\nu (t,dx), \\
&u_t\geq S_t \, , \ \ \\ &u_0=\xi\, .\
\end{split}\right.\end{equation}
Here, $S$ is the given obstacle, $a$ is a matrix defining a symmetric operator on an open bounded domain $\cO$, $f,g,h$ are random coefficients.\\
In a recent work \cite{DMZ12} we have proved existence and uniqueness of the solution of equation (\ref{SPDEO}) under standard Lipschitz hypotheses  and $L^2$-type integrability conditions on the coefficients. Let us recall that the solution is  a couple $(u,\nu )$, where $u$ is a process with values in the first order Sobolev space and $\nu$ is a random regular measure forcing $u$ to stay above $S$ and satisfying a minimal Skohorod condition.\\
 In order to give a rigorous meaning to the notion of solution, inspired by the works of M. Pierre in the deterministic case (see \cite{Pierre,PIERRE}), we introduce the notion of parabolic capacity. The key point is that in \cite{DMZ12}, we construct a solution which admits a quasi continuous version hence defined outside a polar set and that regular measures which in general are not absolutely continuous w.r.t. the Lebesgue measure, do not charge polar sets.\\
 There is a huge literature on parabolic SPDE's without obstacle. The study of the $L^p-$norms w.r.t. the randomness of the space-time uniform norm on the trajectories of a stochastic PDE was started by N. V. Krylov in \cite{Krylov}, for a more complete overview of existing works on this subject see \cite{DMS09,DM11} and the references therein. Concerning the obstacle problem, there are two approaches, a probabilistic one (see \cite{MatoussiStoica, Klimsiak}) based on the Feynmann-Kac's formula via the backward doubly stochastic differential equations and the analytical one (see \cite{DonatiPardoux,NualartPardoux,XuZhang}) based on the Green function.\\

To our knowledge, up to now there is no maximum principle result for
quasilinear SPDE with obstacle and even very few results in the
deterministic case. The aim of this paper is to obtain, under
suitable integrability conditions on the coefficients,
$L^p$-estimates for the uniform norm (in time and space) of the
solution, a maximum principle for local solutions of equation
\eqref{SPDEO} and comparison theorems similar to those obtained in
the without obstacle case in \cite{DMS05,DMS09}.  This yields for
example the following result:

\begin{theorem}
Let $(M_t)_{t\geq 0}$ be an Itô process satisfying some
integrability conditions, $p\geq 2$ and $u$ be a local weak solution
of the obstacle problem (\ref{SPDEO}). Assume that $\partial\cO$ is
Lipschitz and $u\leq M$ on $\partial\cO$,
 then for all $t\in [0,T]$: $$ E\left\| \left( u-M\right)
^{+}\right\| _{\infty ,\infty ;t}^p\le k\left( p,t\right)
\mathcal{C}(S, f, g, h, M)$$ where $\mathcal{C}(S, f, g, h, M)$
depends only on the barrier $S$, the initial condition $\xi$,
coefficients $f,g,h$, the boundary condition  $ M$ and $k$ is a
function which only depends on $p$ and $t$, $\Vert \cdot \Vert
_{\infty ,\infty ;t}$ is the uniform norm on $[0,t]\times {\cal
O}$.\end{theorem}
Let us remark that in order to get such a result, we define the notion of { local solutions} to the obstacle problem \eqref{SPDEO} and so introduce what we call {\it local regular measures}. \\

The paper is organized as follows: in section 2 we introduce
notations and hypotheses. In section 3, we establish the
$L^p-$estimate for uniform norm of the solution with null Dirichlet
boundary condition. Section 4 is devoted to the main result: the
maximum principle for local solutions whose proof is based on an
It\^o formula satisfied by the positive part of any local solution
with lateral boundary condition, $M$. The last section is an
Appendix in which we give  the proofs of several
lemmas.\section{Preliminaries}
\subsection{$L^{p,q}-$space}
Let $\mathcal{O}\subset \bbR^d$ be an open bounded domain  and
$L^2(\mathcal{O})$ the set of square integrable functions with
respect to the Lebesgue measure on $\mathcal{O}$, it is an Hilbert
space equipped with the usual scalar product and norm as follows
$$(u,v)=\int_\mathcal{O}u(x)v(x)dx,\qquad\parallel u\parallel=(\int_\mathcal{O}u^2(x)dx)^{1/2}.$$
In general, we shall extend the notation
$$(u,v)=\int_\mathcal{O}u(x)v(x)dx,$$ where $u,\ v$ are measurable
functions defined on $\cO$ such that $uv\in L^1(\cO)$.
\\The first order Sobolev space of functions vanishing at the boundary will be denoted by $H_0^1(\cO)$,
its natural scalar product and norm are $$
\left( u,v\right) _{H_0^1\left( {\cal O}\right) }=\left( u,v\right) +\int_{%
{\cal O}}\sum_{i=1}^d\left( \partial _iu\left( x\right) \right)
\left(
\partial _iv\left( x\right) \right) dx,\;\left\| u\right\| _{H_0^1\left(
{\cal O}\right) }=\left( \left\| u\right\| _2^2+\left\| \nabla
u\right\| _2^2\right) ^{\frac 12}. $$As usual we shall denote
$H^{-1}(\cO)$ its dual space. \\We shall denote  by  $ H_{loc}^1
(\cO)$ the space of functions which are locally square integrable in
$ \mathcal{O}$ and which admit first order derivatives that are also
locally square integrable.

For each $t>0$ and for all real numbers $p,\,q\geq 1$, we denote by $%
L^{p,q}([0,t]\times {\cal O})$ the space of (classes of) measurable
functions $u:[0,t]\times {\cal O}\longrightarrow \mathbb{{R}}$ such
that
$$
\Vert u\Vert _{p,q;\,t}:=\left( \int_0^t\left( \int_{{\cal O}%
}|u(s,x)|^p\,dx\right) ^{q/p}\,ds\right) ^{1/q} $$ is finite. The
limiting cases with $p$ or $q$ taking the value $\infty $ are also
considered with the use of the essential sup norm.
%We identify this space, in an obvious way, with the space $L^q\left( \left[0,t\right];L^p\left( {\cal O}\right) \right) ,$ consisting of all measurable functions
%$u:\left[ 0,t\right] \rightarrow L^p\left( {\cal O}\right) $ such that $\displaystyle \int_0^t\left\| u_s\right\| _p^qds<\infty .$ This
%identification implies that $\displaystyle \left( \int_0^t\left\|u_s\right\| _p^qds\right) ^{\frac 1q}=\Vert u\Vert _{p,q;\,t}.$
\\Now we introduce some other spaces of functions and discuss a certain duality between them. Like in \cite{DMS05} and \cite{DMS09},
for self-containeness, we recall the following definitions:
\\Let $(p_1,q_1)$, $(p_2,q_2)$ $\in[1,\infty]^2$ be fixed and set $$ I=I\left( p_1,q_1,p_2,q_2\right) :=\left\{ \left( p,q\right)
\in \left[ 1,\infty \right] ^2/\;\exists \;\rho \in \left[
0,1\right] s.t.\right. $$ $$ \left. \frac 1p=\rho \frac
1{p_1}+\left( 1-\rho \right) \frac 1{p_2},\frac 1q=\rho \frac
1{q_1}+\left( 1-\rho \right) \frac 1{q_2}\right\} . $$ This means
that the set of inverse pairs $\left( \frac 1p,\frac 1q\right) ,$
$(p,q)$
 belonging to $I,$ is a segment contained in the square $%
\left[ 0,1\right] ^2,$ with the extremities $\left( \frac
1{p_1},\frac 1{q_1}\right) $ and $\left( \frac 1{p_2},\frac
1{q_2}\right) .$ \\We introduce: $$ L_{I;t}=\bigcap_{\left(
p,q\right) \in I}L^{p,q}\left( \left[ 0,t\right] \times {\cal
O}\right) . $$ We know that this space coincides with the
intersection of the extreme spaces,
$$ L_{I;t}=L^{p_1,q_1}\left( \left[ 0,t\right] \times {\cal
O}\right) \cap L^{p_2,q_2}\left( \left[ 0,t\right] \times {\cal
O}\right) $$ and that it is a Banach space with the following norm$$
\left\| u\right\| _{I;t}:=\left\| u\right\| _{p_1,q_1;t}\vee \left\|
u\right\| _{p_2,q_2;t}. $$
The other space of interest is the algebraic sum%
$$ L^{I;t}:=\sum_{\left( p,q\right) \in I}L^{p,q}\left( \left[
0,t\right] \times {\cal O}\right) , $$ which represents the vector
space generated by the same family of spaces.
This is a normed vector space with the norm%
$$ \left\| u\right\|^{I;t} :=\,\inf \left\{ \sum_{i=1}^n\left\|
u_i\right\| _{p_i,q_i;\,t}\,/\;u=\sum_{i=1}^nu_i,u_i\in
L^{p_i,q_i}\left( \left[ 0,t\right] \times {\cal O}\right) ,\,
\left( p_i,q_i\right) \in I,\,i=1,...n;\,n\in
\mathbb{{N}}^{*}\right\} . $$
Clearly one has $L^{I;t}\subset L^{1,1}\left( \left[ 0,t\right] \times {\cal %
O}\right) $ and $\left\| u\right\| _{1,1;t}\le c\left\| u\right\|
^{I;t},$ for each $u\in L^{I;t},$ with a certain constant $c>0.$

We also remark that if $\left( p,q\right) \in I,$ then the conjugate pair $%
\left( p^{\prime },q^{\prime }\right) ,$ with $\frac 1p+\frac
1{p^{\prime }}=\frac 1q+\frac 1{q^{\prime }}=1,$ belongs to another
set, $I^{\prime },$
of the same type. This set may be described by%
$$ I^{\prime }=I^{\prime }\left( p_1,q_1,p_2,q_2\right) :=\left\{
\left( p^{\prime },q^{\prime }\right) /\;\exists \left( p,q\right)
\in I\;s.t.\;\frac 1p+\frac 1{p^{\prime }}=\frac 1q+\frac
1{q^{\prime }}=1\right\} $$ and it is not difficult to check that
$I^{\prime }\left( p_1,q_1,p_2,q_2\right) =I\left( p_1^{\prime
},q_1^{\prime },p_2^{\prime },q_2^{\prime }\right) ,$ where
$p_1^{\prime },q_1^{\prime },p_2^{\prime }$ and $q_2^{\prime }$ are
defined by $\frac 1{p_1}+\frac 1{p_1^{\prime }}=\frac 1{q_1}+\frac
1{q_1^{\prime }}=\frac 1{p_2}+\frac 1{p_2^{\prime }}=\frac
1{q_2}+\frac 1{q_2^{\prime }}=1.$

Moreover, by H\"older's inequality, it follows that one has
\begin{equation}
\label{dual}\int_0^t\int_{{\cal O}}u\left( s,x\right) v\left(
s,x\right) dxds\le \left\| u\right\| _{I;t}\left\| v\right\|
^{I^{\prime };t},
\end{equation}
for any $u\in L_{I;t}$ and $v\in L^{I^{\prime };t}.$ This inequality
shows
that the scalar product of $L^2\left( \left[ 0,t\right] \times {\cal O}%
\right) $ extends to a duality relation for the spaces $L_{I;t}$ and
$L^{I^{\prime };t}.$

Now let us recall that the Sobolev inequality states that%
\begin{equation}\label{Sobolev}
\left\| u\right\| _{2^{*}}\le c_S\left\| \nabla u\right\| _2,
\end{equation}
for each $u\in H_0^1\left( {\cal O}\right) ,$ where $c_S>0$ is a
constant
that depends on the dimension and $2^{*}=\frac{2d}{d-2}$ if $d>2,$ while $%
2^{*}$ may be any number in $]2,\infty [$ if $d=2$ and $2^{*}=\infty $ if $%
d=1.$ Therefore one has%
$$ \left\| u\right\| _{2^{*},2;t}\le c_S\left\| \nabla u\right\|
_{2,2;t}, $$ for each $t\ge 0$ and each $u\in L_{loc}^2\left(
\mathbb{R}_{+};H_0^1\left( {\cal O}\right) \right) .$ If $u\in
L_{loc}^{\infty}\left( \mathbb{R}_{+}; L^2\left( {\cal O}\right) \,
\right) \bigcap L^2_{loc} \left( \mathbb{R}_+;  H_0^1\left( {\cal
O}\right) \right),$ one has
$$\left\| u\right\| _{2,\infty ;t}\vee \left\| u\right\|
_{2^{*},2;t}\le c_1\left( \left\| u\right\| _{2,\infty ;t}^2+\left\|
\nabla u\right\| _{2,2;t}^2\right) ^{\frac 12},$$with $c_1=c_S\vee
1.$

%One particular case of interest for us in relation with this
%inequality is when $p_1=2,q_1=\infty $ and $p_2=2^{*},q_2=2.$ If
%$I=I\left( 2,\infty ,2^{*},2\right) ,$ then the corresponding set of
%associated conjugate numbers is $I^{\prime }=I^{\prime }\left(
%2,\infty ,2^{*},2\right) =I\left( 2,1,\frac{2^{*}}{2^{*}-1},2\right)
%,$ where for $d=1$ we make the convention that
%$\frac{2^{*}}{2^{*}-1}=1.$ In this particular case we shall use the
%notation $L_{\#;t}:=L_{I;t}$ and $L_{\#;t}^*:=L^{I^{\prime };t}$ and
%the
%respective norms will be denoted by%
%$$ \left\| u\right\| _{\#;t}:=\left\| u\right\| _{I;t}=\left\|
%u\right\| _{2,\infty ;t}\vee \left\| u\right\|
%_{2^{*},2;t},\;\left\| u\right\|_{\#;t}^*:=\left\| u\right\|
%^{I^{\prime };t}. $$ Thus we may write
%\begin{equation}\label{sobolev}\
%\left\| u\right\| _{\#;t}\le c_1\left( \left\| u\right\| _{2,\infty
%;t}^2+\left\| \nabla u\right\| _{2,2;t}^2\right) ^{\frac 12},
%\end{equation}
%for any  $u\in L_{loc}^{\infty}\left( \mathbb{{R}}_{+}; L^2\left(
%{\cal O}\right) \, \right) \bigcap L^2_{loc} \left( \mathbb{R}_+;
%H_0^1\left( {\cal O}\right) \right)$ and $t\ge 0$ and the
%duality inequality becomes%

For $d\ge 3$ and some parameter $\theta \in [0,1[$ we set:
$$
\Gamma_\theta =\left\{ \left( p,q\right) \in  \left[ 1,\infty
\right]^2, \,  \frac d{2p}+\frac 1q=\frac d2+\theta \right\} ,
$$

$$ \Gamma _\theta ^{*}=\left\{ \left( p,q\right) \in \left[
1,\infty \right] ^2/\;\frac d{2p}+\frac 1q=1-\theta \right\} , $$
$$ L_\theta ^{*}=\sum_{\left( p,q\right) \in \Gamma _\theta
^{*}}L^{p,q}\left( \left[ 0,t\right] \times {\cal O}\right) $$ $$
\left\| u\right\| _{\theta ;t}^{*}:=\,\inf \left\{
\sum_{i=1}^n\left\| u_i\right\|
_{p_i,q_i;\,t}\,/\;u=\sum_{i=1}^nu_i,u_i\in L^{p_i,q_i}\left( \left[
0,t\right] \times {\cal O}\right) ,\right. $$ $$
\left. \left( p_i,q_i\right) \in \Gamma _\theta ^{*},\,i=1,...n;\,n\in {\bf N%
}^{*}\right\} . $$ If $d=1,2.$ we put
$$
\Gamma _\theta =\left\{ \left( p,q\right) \in \left[ 1,\infty \right] ^2/\;%
\frac{2^{*}}{2^{*}-2}\frac 1p+\frac 1q=\frac{2^{*}}{2^{*}-2}+\theta
\right\} , $$ $$ \Gamma _\theta ^{*}=\left\{ \left( p,q\right) \in
\left[ 1,\infty \right] ^2/\;\frac{2^{*}}{2^{*}-2}\frac 1p+\frac
1q=1-\theta \right\} $$
and by using similar calculations with the convention $%
\frac{2^{*}}{2^{*}-2}=1$ if $d=1.$

We remark that $%
\Gamma _\theta ^{*}=I\left( \infty ,\frac 1{1-\theta },\frac
d{2\left( 1-\theta \right) },\infty \right) $ and that the norm
$\left\| u\right\| _{\theta ;t}^{*}$ coincides with $\left\|
u\right\| ^{\Gamma _\theta ^{*};t}=\left\| u\right\| ^{I\left(
\infty ,\frac 1{1-\theta },\frac d{2\left( 1-\theta \right) },\infty
\right) ;t}.$ Moreover we have the following duality relation:
\begin{equation}\label{dual2}\int_0^t\int_{{\cal O}}u\left( s,x\right) v\left( s,x\right)
dxds\le \left\| u\right\| _{\theta;t}\left\| v\right\|_{\theta;t}^*,
\end{equation} for any $u\in L_{\theta;t}$ and $v\in L_{\theta;t}^*$ and the following inequality:
\begin{equation}\label{controltheta}
\left\|u\right\|_{\theta;t}\leq
c_1\left(\left\|u\right\|^2_{2,\infty;t}+\left\|\nabla
u\right\|^2_{2,2;t}\right)^{1/2}.
\end{equation}

\subsection{Hypotheses}
We consider a sequence $((B^i(t))_{t\geq0})_{i\in\mathbb{N}^*}$ of
independent Brownian motions defined on a standard filtered
probability space $(\Omega,\mathcal{F},(\mathcal{F}_t)_{t\geq0},P)$
satisfying the usual conditions.

Let $A$ be a symmetric second order differential operator defined on
the open bounded subset $ \cO \subset \R^d$, with domain
$\mathcal{D}(A)$, given by
$$A:=-L=-\sum_{i,j=1}^d\partial_i(a^{i,j}\partial_j).$$ We assume that
$a=(a^{i,j})_{i,j}$ is a measurable symmetric matrix defined on
$\mathcal{O}$ which satisfies the uniform ellipticity
condition$$\lambda|\xi|^2\leq\sum_{i,j=1}^d
a^{i,j}(x)\xi^i\xi^j\leq\Lambda|\xi|^2,\ \forall x\in\mathcal{O},\
\xi\in \R^d,$$where $\lambda$ and $\Lambda$ are positive constants.
The energy associated with the matrix $a$ will be denoted by
\begin{equation}
\label{energy}
 \mathcal{E} \left( w,v\right)=\sum_{i,j=1}^d
\int_{\cO}a^{i,j}(x)\partial_i w(x)\partial_j v(x)\, dx .
\end{equation} It's defined for functions $w,\, v \in  H^1_{0}
(\mathcal{O} )$, or for $w \in H^1_{loc} (\mathcal{O} )$ and $v \in
H^1_{0} (\mathcal{O })$ with compact support.

We assume that we have predictable random
functions\begin{eqnarray*}&&f:\R_+\times\Omega\times\mathcal{O}\times
\R\times \R^d\rightarrow
\R,\\&&g=(g_1,...,g_d):\R_+\times\Omega\times\mathcal{O}\times
\R\times \R^d\rightarrow
\R^d,\\&&h=(h_1,...,h_i,...):\R_+\times\Omega\times\mathcal{O}\times
\R\times \R^d\rightarrow \R^{\mathbb{N}^*}.\end{eqnarray*} We define
\begin{equation*}
\begin{split}
 &f ( \cdot
,\cdot,\cdot, 0,0):=f^0, \  g( \cdot,\cdot,\cdot ,0,0) :=g^0 =
(g_1^0,...,g_d^0) \ \mbox{and}
\  h( \cdot,\cdot,\cdot ,0,0) :=h^0 = (h_1^0,...,h_i^0,...) .\\
\end{split}
\end{equation*}
In the sequel, $|\cdot|$ will always denote the underlying Euclidean
or $l^2-$norm. For
example$$|h(t,\omega,x,y,z)|^2=\sum_{i=1}^{+\infty}|h_i(t,\omega,x,y,z)|^2.$$
\begin{remark} Let us note that this general setting of the SPDE \eqref{SPDEO} we consider, encompasses the case of an SPDE driven by a space-time noise, colored in space and white in time  as in \cite{SW} for example (see also Example 1 in \cite{DMZ12}).

\end{remark}
\textbf{Assumption (H):} There exist non-negative constants $C,\
\alpha,\ \beta$ such that for almost all $\omega$, the following
inequalities hold for all
$(x,y,z,t)\in\mathcal{O}\times\mathbb{R}\times\mathbb{R}^d\times\mathbb{R}_+$:\begin{enumerate}
\item   $|f(t,\omega,x,y,z)-f(t,\omega,x,y',z')|\leq C(|y-y'|+|z-z'|),$
\item $|g(t,\omega,x,y,z)-g(t,\omega,x,y',z')|\leq
C|y-y'|+\alpha|z-z'|,$
\item $|h(t,\omega,x,y,z)-h(t,\omega,x,y',z')|\leq
C|y-y'|+\beta|z-z'|,$
\item the contraction property: $2\alpha+\beta^2<2\lambda.$
\end{enumerate}

Moreover we  introduce some integrability conditions  on the
coefficients $f^0, \;
g^0, \, h^0$ and the initial data $\xi$. Along this article, we fix a terminal time $T>0$. \\

 \textbf{Assumption (HI2)}   $$  E \left(\|\xi \|_2^2   +\left\| f^0\right\|_{2,2;t}^2+\left\|
\left|g^0\right|\right\| _{2,2;t}^2+\left\|\left| h^0\right|\right\|
_{2,2;t}^2\right) <\infty , $$ for each $t\in[0,T]$.
\\[0.2cm]
\textbf{Assumption (HIL)} $$  E \int_K |\xi (x)|^2 dx + E \,
\int_{0}^t \int_K \big( |f_s^0(x)|^2 + |g_s^0 (x)|^2 + |h_s^0(x)|^2
\, \big) dx ds < \infty,
$$for any compact set $K \subset \mathcal{O}$ and for any $t\in[0,T]$.

\subsection{Weak solutions}
We now introduce $\cH_T$, the space of
  $H_0^1(\mathcal{O})$-valued
  predictable processes $(u_t)_{t \in[0,T]}$ such that
\[
\left( E \sup_{0\leq s\leq T} \left\| u_{s}\right\|_2
^{2}+\int_{0}^{T}E\, \mathcal{E}\left( u_{s}\right) ds\right)
^{1/2}\;< \; \infty \;.
\]
We define $\HH_{loc}=\HH_{loc}(\mathcal{O})$ to be the set of
$H^1_{loc} (\OO )$-valued predictable processes defined on $[0,T]$
such that for any compact subset $K$ in $\OO$:

\[
\left( E \sup_{0\leq s\leq T}\int_K u_{s}(x)^{2}\, dx
+E\int_{0}^{T}\int_K |\nabla u_{s}(x)|^2\, dx ds\right) ^{1/2}\;< \;
\infty .
\]

The space of test functions is the algebraic tensor product
$\mathcal{D}=\mathcal{C} _{c}^{\infty }(\bbR^+)\otimes
\mathcal{C}_c^2 (\cO )$, where $\mathcal{C} _{c}^{\infty }(\bbR^+)$
denotes the space of all real infinite differentiable  functions
with compact support in $\mathbb{R}^+$ and $\mathcal{C}_c^2 (\cO )$
the set of $C^2$-functions with compact support in $\cO$.

\vspace{0.5cm}

Now we recall the definition of the regular measure which has been defined in \cite{DMZ12}.\\
$\mathcal{K}$ denotes $L^\infty([0,T];L^2(\mathcal{O}))\cap
L^2([0,T];H_0^1(\mathcal{O}))$ equipped with the norm:
\begin{eqnarray*}\parallel
v\parallel^2_\mathcal{K}&=&\parallel
v\parallel^2_{L^\infty([0,T];L^2(\mathcal{O}))}+\parallel
v\parallel^2_{L^2([0,T];H_0^1(\mathcal{O}))}\\
&=&\sup_{t\in[0,T[}\parallel v_t\parallel^2 +\int_0^T \left(
\parallel v_t \parallel^2  +\mathcal{E}(v_t)\right)\, dt
.\end{eqnarray*} $\mathcal{C}$ denotes the space of continuous
functions with compact support in $[0,T[\times\mathcal{O}$ and
finally:
$$\mathcal {W}=\{\varphi\in L^2([0,T];H_0^1(\mathcal{O}));\ \frac{\partial\varphi}{\partial t}\in
L^2([0,T];H^{-1}(\mathcal{O}))\}, $$ endowed with the
norm$\parallel\varphi\parallel^2_{\mathcal {W}}=\parallel
\varphi\parallel^2_{L^2([0,T];H_0^1(\mathcal{O}))}+\parallel\displaystyle\frac{\partial
\varphi}{\partial t}\parallel^2_{L^2([0,T];H^{-1}(\mathcal{O}))}$. \\
It is known (see \cite{LionsMagenes}) that $\mathcal{W}$ is
continuously embedded in $C([0,T]; L^2 (\cO))$, the set of $L^2 (\cO
)$-valued continuous functions on $[0,T]$. So without ambiguity, we
will also consider
$\mathcal{W}_T=\{\varphi\in\mathcal{W};\varphi(T)=0\}$,
$\mathcal{W}^+=\{\varphi\in\mathcal{W};\varphi\geq0\}$,
$\mathcal{W}_T^+=\mathcal{W}_T\cap\mathcal{W}^+$.
\begin{definition}
An element $v\in \mathcal{K}$ is said to be a {\bf parabolic
potential} if it satisfies:
$$ \forall\varphi\in\mathcal{W}_T^+,\
\int_0^T-(\frac{\partial\varphi_t}{\partial
t},v_t)dt+\int_0^T\mathcal{E}(\varphi_t,v_t)dt\geq0.$$ We denote by
$\mathcal{P}$ the set of all parabolic potentials.
\end{definition}
The next representation property is  crucial:
\begin{proposition}(Proposition 1.1 in \cite{PIERRE})\label{presentation}
Let $v\in\mathcal{P}$, then there exists a unique positive Radon
measure on $[0,T[\times\mathcal{O}$, denoted by $\nu^v$, such that:
$$\forall\varphi\in\mathcal{W}_T\cap\mathcal{C},\ \int_0^T(-\frac{\partial\varphi_t}{\partial t},v_t)dt+\int_0^T\mathcal{E}(\varphi_t,v_t)dt=\int_0^T\int_\mathcal{O}\varphi(t,x)d\nu^v.$$
Moreover, $v$ admits a right-continuous (resp. left-continuous)
version $\hat{v} \ (\makebox{resp. } \bar{v}): [0,T]\mapsto L^2
(\cO)$ .\\
Such a Radon measure, $\nu^v$ is called {\bf a regular measure} and
we write:
$$ \nu^v =\frac{\partial v}{\partial t}+Av .$$
\end{proposition}

\begin{definition}
Let $K\subset [0,T[\times\mathcal{O}$ be compact, $v\in\mathcal{P}$
is said to be  \textit{$\nu-$superior} than 1 on $K$, if there
exists a sequence $v_n\in\mathcal{P}$ with $v_n\geq1\ a.e.$ on a
neighborhood of $K$ converging to $v$ in
$L^2([0,T];H_0^1(\mathcal{O}))$.
\end{definition}
We denote:$$\mathscr{S}_K=\{v\in\mathcal{P};\ v\ is\ \nu-superior\
to\ 1\ on\ K\}.$$
\begin{proposition}(Proposition 2.1 in \cite{PIERRE})
Let $K\subset [0,T[\times\mathcal{O}$ compact, then $\mathscr{S}_K$
admits a smallest $v_K\in\mathcal{P}$ and the measure $\nu^v_K$
whose support is in $K$ satisfies
$$\int_0^T\int_\mathcal{O}d\nu^v_K=\inf_{v\in\mathcal{P}}\{\int_0^T\int_\mathcal{O}d\nu^v;\ v\in\mathscr{S}_K\}.$$
\end{proposition}
\begin{definition}(Parabolic Capacity)\begin{itemize}
                                        \item Let $K\subset [0,T[\times\mathcal{O}$ be compact, we define
$cap(K)=\int_0^T\int_\mathcal{O}d\nu^v_K$;
                                        \item let $O\subset
[0,T[\times\mathcal{O}$ be open, we define $cap(O)=\sup\{cap(K);\
K\subset O\ compact\}$;
                                        \item   for any borelian
$E\subset [0,T[\times\mathcal{O}$, we define $cap(E)=\inf\{cap(O);\
O\supset E\ open\}$.
                                      \end{itemize}

\end{definition}
\begin{definition}A property is said to hold quasi-everywhere (in short q.e.)
if it holds outside a set of null capacity.
\end{definition}
\begin{definition}(Quasi-continuous)

\noindent A function $u:[0,T[\times\mathcal{O}\rightarrow\mathbb{R}$
 is called quasi-continuous, if there exists a decreasing sequence of open
subsets $O_n$ of $[0,T[\times\mathcal{O}$ with: \begin{enumerate}
                        \item for all $n$, the restriction of $u_n$ to the complement of $O_n$ is
continuous;
                                       \item $\lim_{n\rightarrow+\infty}cap\;(O_n)=0$.
                                     \end{enumerate}
We say that $u$ admits a quasi-continuous version, if there exists
$\tilde{u}$ quasi-continuous  such that $\tilde{u}=u\ a.e.$
\end{definition}
The next proposition, whose proof may be found in \cite{Pierre} or
\cite{PIERRE} shall play an important role in the sequel:
\begin{proposition}\label{Versiont} Let $K\subset \cO$ a compact set, then $\forall t\in [0,T[$
$$cap (\{ t\}\times K)=\lambda_d (K),$$
where $\lambda_d$ is the Lebesgue measure on $\cO$.\\
As a consequence, if $u: [0,T[\times \cO\rightarrow \R$ is a map
defined quasi-everywhere then it defines uniquely a map from $[0,T[$
into $L^2 (\cO)$. In other words, for any $t\in [0,T[$, $u_t$ is
defined without any ambiguity as an element in $L^2 (\cO)$.
Moreover, if $u\in \mathcal{P}$, it admits  version $\bar{u}$ which
is left continuous on $[0,T]$ with values in $L^2 (\cO )$ so that
$u_T =\bar{u}_{T^-}$ is also defined without ambiguity.
\end{proposition}
\begin{remark} The previous proposition applies if for example $u$ is quasi-continuous.
\end{remark}
To establish a maximum principle for local solutions we need to
define the notion of {\it local regular measures}:
\begin{definition}
We say that a Radon measure $\nu$ on $[0,T[\times \cO$ is a {local
regular measure} if for any  non-negative $\phi$  in
$\mathcal{C}_c^\infty(\cO)$,  $\phi\nu$ is a regular measure.
\end{definition}
\begin{proposition}
Local regular measures do not charge polar sets (i.e. sets of
capacity 0).
\end{proposition}
\begin{proof}
Let $A$ be a polar set and consider a sequence $(\phi_n)$ in
$\mathcal{C}_c^\infty(\cO)$, $0\leq\phi_n\leq1$, converging to 1
everywhere on $\cO$. By Fatou's lemma,
\begin{equation*}
0\leq\int_{[0,T[\times\cO}\1_A
d\nu(x,t)\leq\liminf_{n\rightarrow\infty}\int_{[0,T[\times\cO}\1_A\phi_nd\nu(x,t)=0.
\end{equation*}
\end{proof}
%Now we introduce that, $\mathcal{K}_{loc}$ denotes
%$L^\infty(0,T;L^2_{loc}(\mathcal{O}))\cap
%L^2(0,T;H_{loc}^1(\mathcal{O}))$ equipped with the norm:
%\begin{eqnarray*}\parallel
%v\parallel^2_{\mathcal{K}_{loc}}&=&\parallel
%v\parallel^2_{L^\infty(0,T;L^2_{loc}(\mathcal{O}))}+\parallel
%v\parallel^2_{L^2(0,T;H_{loc}^1(\mathcal{O}))}\\
%&=&\sup_{t\in[0,T[}\int_K(v_t(x))^2dx +\int_0^T\int_K \left(
%v_t(x)^2  +|\nabla v_t(x)|^2\right)dx dt .\end{eqnarray*}
%\begin{definition}
%An element $v\in \mathcal{K}_{loc}$ is said to be a {\bf local
%parabolic potential} if it satisfies:
%$$ \forall\varphi\in\mathcal{W}_T^+,\
%\int_0^T-(\frac{\partial\varphi_t}{\partial
%t},v_t)dt+\int_0^T\mathcal{E}(\varphi_t,v_t)dt\geq0.$$ We denote by
%$\mathcal{P}_{loc}$ the set of all parabolic potentials.
%\end{definition}
%\textbf{We define $\nu$ is a local regular measure associated to a
%local parabolic potential $v$ if it satisfies the relation in
%Proposition \ref{presentation}. }

We end this part by a convergence lemma which plays an important
role in our approach (Lemma 3.8 in \cite{PIERRE}):
\begin{lemma}\label{convergemeas}
If $v^n\in\mathcal{P}$ is a bounded sequence in $\mathcal{K}$ and
converges weakly to $v$ in $L^2([0,T];H_0^1(\mathcal{O}))$; if $u$
is a quasi-continuous function and $|u|$ is bounded by a element in
$\mathcal{P}$. Then
$$\lim_{n\rightarrow+\infty}\int_0^T\int_\mathcal{O}ud\nu^{v^n}=\int_0^T\int_\mathcal{O}ud\nu^{v}.$$
\end{lemma}
%\begin{remark}For the more general case one can see \cite{PIERRE} Lemma 3.8. \end{remark}
\vspace{0.5cm}

We now give the assumptions on the obstacle that we shall need in the different cases that we shall consider.\\

\textbf{Assumption (O):} The obstacle $S: [0,T]\times \Omega\times
\cO\rightarrow \R$ is  an adapted random field  almost surely
quasi-continuous, in the sense that for $P$-almost all
$\omega\in\Omega$, the map $(t,x)\rightarrow S_t (\omega,x)$ is
quasi-continuous.  Moreover,  $S_0 \leq \xi$ $P$-almost surely and
$S$ is controlled by the solution of an SPDE, i.e. $\forall
t\in[0,T],$
\begin{equation}S_t\leq S'_t,\quad dP\otimes dt\otimes dx-a.e.\end{equation} where
$S'$ is the solution of the linear SPDE
\begin{equation}\left\{\begin{array}{ccl} \label{obstacle}
 dS'_t&=&LS'_tdt+f'_tdt+\sum_{i=1}^d \partial_i g'_{i,t}dt+\sum_{j=1}^{+\infty}h'_{j,t}dB^j_t\\
                  S'(0)&=&S'_0 ,
\end{array}\right. \end{equation}
with null boundary Dirichlet conditions.\\

\textbf{Assumption (OL):} The obstacle $S: [0,T]\times \Omega\times
\cO\rightarrow \R$ is  an adapted random field, almost surely
quasi-continuous, such that $S_0 \leq \xi$ $P$-almost surely and
controlled by a \textbf{local} solution of an SPDE, i.e. $\forall
t\in[0,T],$
\begin{equation*}S_t\leq S'_t,\quad dP\otimes dt\otimes dx-a.e.\end{equation*} where
$S'$ is a \textbf{local} solution of the linear SPDE
\begin{equation*}\left\{\begin{array}{ccl}
 dS'_t&=&LS'_tdt+f'_tdt+\sum_{i=1}^d \partial_i g'_{i,t}dt+\sum_{j=1}^{+\infty}h'_{j,t}dB^j_t\\
                  S'(0)&=&S'_0 .
\end{array}\right. \end{equation*}
  \textbf{Assumption (HO2)}   $$  E \left(\|\xi \|_2^2   +\left\| f'\right\|_{2,2;T}^2+\left\|
|g'|\right\| _{2,2;T}^2+\left\| |h'|\right\| _{2,2;T}^2\right)
<\infty .
$$
\\[0.2cm]
\textbf{Assumption (HOL)} $$  E \int_K |S'_0|^2 dx + E \, \int_{0}^T
\int_K \big( |f'_t(x)|^2 + |g'_t (x)|^2 + |h'_t (x)|^2 \, \big) dx
dt < \infty
$$for any compact set $K \subset \mathcal{O}$.\begin{remark}\label{remark3} It is well-known that under {\bf (HO2)} $S'$ belongs to $\cH_T$, is unique and satisfies the following
estimate:
\begin{equation}\label{estimobstacle1}
E\sup_{t\in[0,T]}\parallel
S'_t\parallel^2+E\int_0^T\mathcal{E}(S'_t)dt\leq CE\left[\parallel
S'_0\parallel^2+\int_0^T(\parallel f'_t\parallel^2+\parallel
|g'_t|\parallel^2+\parallel |h'_t|\parallel^2)dt\right],
\end{equation}
see for example Theorem 8 in \cite{DenisStoica}. Moreover, as a
consequence of  Theorem 3 in \cite{DMZ12}, we know that $S'$ admits
a quasi-continuous version.
\end{remark}
\begin{definition} A pair
$(u,\nu)$ is said to be a solution of the problem (\ref{SPDEO}) if
\begin{enumerate}
    \item $u\in\mathcal{H}_T$, $u(t,x)\geq S(t,x),\ dP\otimes dt\otimes
    dx-a.e.$ and $u_0(x)=\xi,\ dP\otimes dx-a.e.$;
    \item $\nu$ is a random regular measure defined on
    $[0,T[\times\mathcal{O}$;
    \item the following relation holds almost surely, for all
    $t\in[0,T]$ and all $\varphi\in\mathcal{D}$,
     \begin{equation}\begin{split}\label{solution}(u_t,\varphi_t)=&(\xi,\varphi_0)+\int_0^t(u_s,\partial_s\varphi_s)ds-\int_0^t\mathcal{E}(u_s,\varphi_s)ds\\&-\sum_{i=1}^d\int_0^t(g^i_s(u_s,\nabla u_s),\partial_i\varphi_s)ds
    +\int_0^t(f_s(u_s,\nabla u_s),\varphi_s)ds\\&+\sum_{j=1}^{+\infty}\int_0^t(h^j_s(u_s,\nabla u_s),\varphi_s)dB^j_s+\int_0^t\int_{\mathcal{O}}\varphi_s(x)\nu(dx,ds);\end{split}\end{equation}
    \item $u$ admits a quasi-continuous version, $\tilde{u}$, and we have  $$\int_0^T\int_\cO(\tilde{u}(s,x)-S(s,x))\nu(dx,ds)=0,\ \
    P-a.s.$$
  \end{enumerate}
\end{definition}
We denote by $\cR(\xi,f,g,h,S)$ the solution of the obstacle problem
when it exists and  is unique.
\begin{definition} A pair
$(u,\nu)$ is said to be a local solution of the problem
(\ref{SPDEO}) if
\begin{enumerate}
    \item $u\in\mathcal{H}_{loc}$, $u(t,x)\geq S(t,x),\ dP\otimes dt\otimes
    dx-a.e.$ and $u_0(x)=\xi,\ dP\otimes dx-a.e.$;
    \item $\nu$ is a local random regular measure defined on
    $[0,T[\times\mathcal{O}$;
    \item the following relation holds almost surely, for all
    $t\in[0,T]$ and all $\varphi\in\mathcal{D}$,
     \begin{equation}\begin{split}\label{solutionlocal}(u_t,\varphi_t)=&(\xi,\varphi_0)+\int_0^t(u_s,\partial_s\varphi_s)ds-\int_0^t\mathcal{E}(u_s,\varphi_s)ds\\&-\sum_{i=1}^d\int_0^t(g^i_s(u_s,\nabla u_s),\partial_i\varphi_s)ds
    +\int_0^t(f_s(u_s,\nabla u_s),\varphi_s)ds\\&+\sum_{j=1}^{+\infty}\int_0^t(h^j_s(u_s,\nabla u_s),\varphi_s)dB^j_s+\int_0^t\int_{\mathcal{O}}\varphi_s(x)\nu(dx,ds);\end{split}\end{equation}
     %\begin{eqnarray*}&&(u_t,\varphi_t)-(\xi,\varphi_0)-\int_0^t(u_s,\partial_s\varphi_s)ds+\int_0^t\mathcal{E}(u_s,\varphi_s)ds
     %+\sum_{i=1}^d \int_0^t(g^i_s(u_s,\nabla u_s),\partial_i\varphi_s)ds\nonumber\\
    %&&=\int_0^t(f_s(u_s,\nabla u_s),\varphi_s)ds+\sum_{j=1}^{+\infty}\int_0^t(h^j_s(u_s,\nabla u_s),\varphi_s)dB^j_s
    %+\int_0^t\int_{\mathcal{O}}\varphi_s(x)\nu(dx,ds)\end{eqnarray*}
    \item $u$ admits a quasi-continuous version, $\tilde{u}$, and we have  $$\int_0^T\int_\cO(\tilde{u}(s,x)-S(s,x))\nu(dx,ds)=0,\ \
    P-a.s.$$
  \end{enumerate}
\end{definition}
We denote by $\cR_{loc}(\xi,f,g,h,S)$ the set of all the local
solutions $(u,\nu)$.
\\ Finally, in the sequel, we introduce some constants $\epsilon$, $\delta>0$, we shall denote by $C_\epsilon$, $C_\delta$ some constants depending only on  $\epsilon$, $\delta$, typically those appearing in the kind of inequality
\begin{equation}\label{young}|ab|\leq \epsilon a^2 + C_\epsilon b^2.\end{equation}
%$(u,\nu)$ is the solution of the following SPDE with obstacle $S$
%dominated by
%$S'$\begin{equation}\label{OSPDE}du_t+Au_tdt=f_t(u_t,\nabla
%u_t)dt+div g_t(u_t,\nabla u_t)dt+h_t(u_t,\nabla
%u_t)dB_t+\nu(x,dt)\end{equation} and $S'$ satisfies the following
%SPDE \begin{equation}\label{obstacle}dS'_t+AS'_tdt=f'_tdt+div
%g'_tdt+h'_tdB_t\end{equation} Therefore,
%(\ref{OSPDE})-(\ref{obstacle}), $u-S'$ satisfies
%\begin{equation*}d(u_t-S'_t)+A(u_t-S'_t)dt=(f_t-f'_t)dt+div(g_t-g'_t)dt+(h_t-h'_t)dB_t+\nu(x,dt)\end{equation*}

\section{$L^p-$estimate for the uniform norm of solutions with null
Dirichlet boundary condition}\label{LPestimate} In this section, we
want to study, for some $p\geq2$, the $L^p-$ estimate for the
uniform norm  of the solution of  (\ref{SPDEO}). To get such
estimate, we need stronger integrability conditions on the
coefficients and the initial condition. To this end, we consider the
following assumptions: for $\theta\in[0,1[$ and
$p\geq2$:\\

\textbf{Assumption (HI$\mathbf{2
p}$)}$$E\left(\left\|\xi\right\|_\infty^p+\left\|
f^0\right\|^2_{2,2;T}+\left\| |g^0 |\right\|^2_{2,2;T}+\left\| |h^0
|\right\|^2_{2,2;T}\right)<\infty .$$

\textbf{Assumption (HO$\mathbf{\infty p}$)}
$$S'_0\in L^\infty(\Omega\times\cO)\ and\ E\left((\left\|
f'\right\|_{\infty,\infty;T})^p+(\left\|  |g'
|^2\right\|_{\infty,\infty;T})^{p/2}+(\left\|  |h'
|^2\right\|_{\infty,\infty;T})^{p/2}\right)<\infty .$$

To get the estimates we need,  we apply It\^o's formula to $u-S'$,
in order to take advantage of the fact that $S-S'$ is non-positive
and that as $u$ is solution of (\ref{SPDEO}) and $S'$ satisfies
(\ref{obstacle}), $u-S'$ satisfies
\begin{equation}\label{uminusS'}\left\{ \begin{split}&d(u_t-S'_t)=\partial_i  (a_{i,j}(x)\partial_j(u_t(x)-S'_t(x)))dt+(f(t,x,u_t(x),\nabla
u_t(x))-f'(t,x))dt \\&+\partial_i(g_i(t,x,u_t(x),\nabla
u_t(x))-g'_i(t,x))dt+(h_j(t,x,u_t(x),\nabla
u_t(x))-h'_j(t,x))dB^j_t\\&+\nu(x,dt), \\
 &(u-S')_0=\xi-S'_0\, ,\\&u-S'\geq S-S'\, . \end{split}\right.\end{equation}
that is why we introduce the following functions:
$$\bar{f}(t,\omega,x,y,z)=f(t,\omega,x,y+S'_t,z+\nabla
S'_t)-f'(t,\omega,x)$$
$$\bar{g}(t,\omega,x,y,z)=g(t,\omega,x,y+S'_t,z+\nabla S'_t)-g'(t,\omega,x)$$
$$\bar{h}(t,\omega,x,y,z)=h(t,\omega,x,y+S'_t,z+\nabla S'_t)-h'(t,\omega,x).$$
Let us remark that the Skohorod condition for $u-S'$ is satisfied
since
$$ \int_0^T\int_{\cO} (u_s (x)-S'_s (x))-(S_s (x)-S'_s (x))\nu (ds ,dx)=\int_0^T\int_{\cO} (u_s (x)-S_s (x))\nu (ds ,dx)=0.$$
It is obvious that $\bar{f}$, $\bar{g}$ and $\bar{h}$ satisfy the
Lipschitz conditions with the same Lipschitz coefficients as $f$,
$g$ and $h$ and $\left\|\xi-S'_0\right\|_\infty\in
L^p(\Omega,P)$. Nevertheless, we need a supplementary hypothesis:\\
\textbf{Assumption (HD$\mathbf{\theta p}$)} $$E((\left\|
\bar{f}^0\right\|^*_{\theta;T})^p+(\left\|  |\bar{g}^0
|^2\right\|^*_{\theta;T})^{p/2}+(\left\|  |\bar{h}^0
|^2\right\|^*_{\theta;T})^{p/2})<\infty .$$

This assumption is fulfilled in the following case:
\begin{example}
If $\left\|\nabla S'\right\|^*_{\theta;T},\
\left\|f^0\right\|^*_{\theta;T},\ \left\|g^0\right\|^*_{\theta;T}\
and\ \left\|h^0\right\|^*_{\theta;T}$ belong to $ L^p(\Omega,P), $
and assumptions {\bf (H)} and  {\bf(HO$\mathbf{\infty p}$)}  hold,
then:\vspace{0.5cm}
\\$\bar{f}$ satisfies the Lipschitz condition with the same Lipschitz coefficients:
 \begin{eqnarray*}\left|\bar{f}(t,\omega,x,y,z)-\bar{f}(t,\omega,x,y',z')\right|&=&\big|f(t,\omega,x,y+S'_t(x),z+\nabla S'_t(x))+f'(t,\omega,x)\\&-&f(t,\omega,x,y'+S'_t(x),z'+\nabla S'_t(x))-f'(t,\omega,x)\big|\\&\leq& C\left|y-y'\right|+C\left|z-z'\right|.\end{eqnarray*}

$\bar{f}$ satisfies the integrability condition:
\begin{eqnarray*}\left\|\bar{f}^0\right\|^*_{\theta;T}&=&\left\|f(S',\nabla
S')-f'\right\|^*_{\theta;T}\leq \left\|f(S',\nabla
S')\right\|^*_{\theta;T}+\left\|f'\right\|^*_{\theta;T}\\&\leq&\left\|f^0\right\|^*_{\theta;T}+C\left\|S'\right\|^*_{\theta;T}+C\left\|\nabla
S'\right\|^*_{\theta;T}+\left\|f'\right\|_{\infty,\infty;T}.\end{eqnarray*}

And the same for $\bar{g}$ and $\bar{h}$, which proves that {\bf
(HD$\mathbf{\theta p}$)} holds.
\end{example}
We now give the main result of this Section, which is a version of
the maximum principle in the case of a solution vanishing on the
boundary of $\cO$:
\begin{theorem}\label{LPESTIM}
Suppose that assumptions {\bf (H)}, {\bf (O)}, {\bf (HI$\mathbf{2
p}$)}, {\bf (HO$\mathbf{\infty p}$)} and  {\bf (HD$\mathbf{\theta
p}$)} hold, for some $\theta\in[0,1[$ and $p\geq2$ and that the
constants of Lipschitz conditions satisfy
$$\alpha+\frac{\beta^2}{2}+72\beta^2<\lambda.$$ Let $(u,\ \nu)$ be
the solution of OSPDE (\ref{SPDEO}) with null boundary condition,
then for all $t\in [0,T]$,
\begin{equation*}\begin{split}E\left\| u\right\|^p_{\infty,\infty;t}&\leq c(p)k(t)E\big(\left\|\xi\right\|_\infty^p+\left\|S'_0\right\|_\infty^p+\left\|f'\right\|_{\theta;t}^{*p}+\left\||g'|^2\right\|_{\theta;t}^{*p/2}+\left\||h'|^2\right\|_{\theta;t}^{*p/2}\\&\quad+\left\|\bar{f}^0\right\|_{\theta;t}^{*p}+\left\||\bar{g}^0|^2\right\|_{\theta;t}^{*p/2}+\left\||\bar{h}^0|^2\right\|_{\theta;t}^{*p/2}\big),\end{split}\end{equation*}
where $c(p)$ is a constant which depends on $p$ and $k(t)$ is a
constant which depends on the structure constants and $t\in[0,T]$.
\end{theorem}

\begin{remark}
The relations
$\left\|f'\right\|_{\theta;t}^{*p}\leq(\left\|f'\right\|_{\infty,\infty;t})^{p}$,
$\left\||g'|^2\right\|_{\theta;t}^{*p/2}\leq(\left\||g'|^2\right\|_{\infty,\infty;t})^{p/2}$
and
$\left\||h'|^2\right\|_{\theta;t}^{*p/2}\leq(\left\||h'|^2\right\|_{\infty,\infty;t})^{p/2}$
and assumption {\bf(HO$\mathbf{\infty p}$))} yield
$$E\left(\left\|f'\right\|_{\theta;t}^{*p}+\left\||g'|^2\right\|_{\theta;t}^{*p/2}+\left\||h'|^2\right\|_{\theta;t}^{*p/2}\right)<+\infty.$$
\end{remark}

As the proof of this theorem is quite long, we split it into several
steps.
\subsection{The case where $\xi$, $\bar{f}^0$, $\bar{g}^0$ and $\bar{h}^0$ are uniformly bounded}\label{casborn}
 In this subsection, we assume that the hypotheses {\bf (H)}, {\bf (O)}, {\bf (HI$\mathbf{2
p}$)}, {\bf (HO$\mathbf{\infty p}$)} hold and  we add the following
stronger ones:
\begin{equation*}\label{stronger1}\xi\in
L^\infty(\Omega\times\cO),\end{equation*}and
$$\bar{f}^0,\ \bar{g}^0,\ \bar{h}^0\in L^\infty(\mathbb{R}_+\times\Omega\times\cO).$$
Then it is obviously that $\xi-S'_0\in L^\infty(\Omega\times\cO)$.\\
Under these hypotheses, we know that the SPDE with obstacle
(\ref{SPDEO}) admits a unique weak solution $(u,\nu) = \mathcal R
(\xi, f,g,h,S)$ and that
$(u-S',\nu)=\cR(\xi-S'_0,\bar{f},\bar{g},\bar{h},S-S')$. We start by
proving the following $L^l-$estimate:

\begin{lemma}\label{estimateul}
The solution $u$ of the problem (\ref{SPDEO}) belongs to
$\cap_{l\geq2}L^l([0,T]\times\cO\times\Omega)$. Moreover there exist
constants $c,\ c'>0$ which only depend on $\ C,\ \alpha,\ \beta$ and
on the quantity
\begin{eqnarray*}K=\left\|\xi-S'_0\right\|_{L^\infty(\Omega\times\cO)}\vee\left\|\bar{f}^0\right\|_{L^\infty(\mathbb{R}_+\times\Omega\times\cO)}\vee\left\|\bar{g}^0\right\|_{L^\infty(\mathbb{R}_+\times\Omega\times\cO)}\vee\left\|\bar{h}^0\right\|_{L^\infty(\mathbb{R}_+\times\Omega\times\cO)}\end{eqnarray*}
 such that, for all real $l\geq2$,
\begin{equation}\label{uS'l}E\int_\cO|u_t(x)-S'_t(x)|^ldx\leq cK^2l(l-1)e^{cl(l-1)t}\end{equation}
\begin{equation}\label{nablauS'l}E\int_0^t\int_\cO|u_s(x)-S'_s(x)|^{l-2}|\nabla(u_s(x)-S'_s(x))|^2dxds\leq c'K^2l(l-1)e^{cl(l-1)t}\end{equation}
and\begin{equation}\label{nuS'l}E\int_0^t\int_\cO
|u_s(x)-S'_s(x)|^{l-1}\nu(dxds)<+\infty.\end{equation}
\end{lemma}
\begin{proof}
Notice first that if
$(u-S',\nu)=\cR(\xi-S'_0,\bar{f},\bar{g},\bar{h},S-S')$, then
$$
\bar{f}\left(u-S',\nabla(u-S')\right)
,\bar{g}_i\left(u-S',\nabla(u-S')\right) ,\bar{h}_i\left(
u-S',\nabla(u-S')\right) \in L^2\left([0,T];L^2\left( \Omega \times
{\cal O}\right) \right)
$$
and consequently we can apply It\^o's formula to $(u-S',\nu)$ (See Theorem 5 in \cite{DMZ12}).\\
 We fix a real $l\geq2$, $T > 0$  and introduce the sequence $(\vphi_n
)_{n\in\Ne}$ of functions such that for all $n\in\Ne$:
\[\forall x\in\R,\, \vphi_n (x)= \left\{
\begin{array}{ll}
 \mid x \mid^l&\makebox{ if }\mid x\mid\leq n\\
 n^{l-2} \, \big[  \frac{l (l-1)}{2} (|x|-n)^2+ l \, n (|x|-n) + n^2\,\big] &\makebox{ if }\mid x\mid >n
 \end{array}\right.\]
One can easily verify that for  fixed $n$, $\vphi_n$ is twice
differentiable with bounded second derivative, $ \varphi_n^{\prime
\prime} (x) \geq 0$, and as $ n \to \infty $ one has
 $ \varphi_n (x)\longrightarrow |x|^{l}$,  $\varphi_n^{\prime} (x)\longrightarrow lsgn(x)|x|^{l-1}$,
  $\varphi_n^{\prime \prime} (x)\longrightarrow l (l-1)|x|^{l-2}$.
  Moreover, the following relations hold, for
all $x\in \bbR$ and $n\geq l$:
\begin{enumerate}
\item $\mid x\vphi_n^{\prime}(x)\mid\leq l\vphi_n (x). $
\item $\mid \vphi_n^{\prime}(x)\mid\leq\mid
x\vphi_n^{\prime\prime}(x)\mid$.
\item $\mid x^2\vphi_n^{\prime\prime}(x)\mid \leq l(l-1)\vphi_n (x)$.
\item $|\vphi_n^{\prime}(x)|\leq l(\vphi_n (x)+1 ).$
\item $|\vphi_n^{\prime\prime}(x)|\leq l(l-1)(\vphi_n (x)+1 ).$
\end{enumerate}
Applying It\^o's formula to $\varphi_n(u-S')$, we have $P$-a.s. for
all $t\in [0,T]$,
\begin{equation}\label{uS'It\^o's}
\begin{split}
& \int_{\mathcal{O}} \varphi_n (u_t (x)-S'_t(x)) \, dx \; + \;
\int_0^t \mathcal{E} \big(\varphi_n^{\prime} (u_s-S'_s), \, u_s-S'_s
\big) \, ds = \int_{\mathcal{O}} \varphi_n (\xi(x)-S'_0(x) ) \, dx
\\ & + \int_{0}^{t}\int_{\mathcal{O}}
\varphi_n^{\prime}(u_s(x)-S'_s(x)) \bar{f}  (s,x, u_s-S'_s,\nabla (u_s-S'_s)) \, dx ds \\
&  - \sum_{i=1}^{d} \int_{0}^{t} \int_{\mathcal{O}}
\varphi_n^{\prime \prime } (u_s(x)-S'_s(x) )
\partial_i (u_s(x)-S'_s(x)) \, \bar{g}_{i} (s,x,u_s-S'_s, \nabla (u_s-S'_s)) \, dx \, ds \\& +
\sum_{j=1}^{\infty} \int_{0}^{t} \int_{\mathcal{O}}
\varphi_n^{\prime} (u_s(x)-S'_s(x)) \,  \bar{h}_{j}( s,x, u_s-S'_s,\nabla (u_s-S'_s))\, dx dB_{s}^{j}\\
& + \frac{1}{2}\, \sum_{j=1}^{\infty} \int_{0}^{t}\int_{\mathcal{O}}
\varphi_n^{\prime \prime } (u_s(x)-S'_s(x)) \bar{h}_{j}^2(s,x,
u_s-S'_s,\nabla (u_s-S'_s)) \, dx\, ds\\&
+\int_0^t\int_\cO\varphi_n^{\prime}(u_s(x)-S'_s(x))\nu(dxds) \, .\\
\end{split}
\end{equation}
Since the support of $\nu$ is $\{u=S\}$, the last term is equal to
$$\int_0^t\int_\cO\varphi_n^{\prime}(S_s(x)-S'_s(x))\nu(dxds)$$and
it is negative, because
$$\int_0^t\int_\cO\varphi_n^{\prime}(S_s(x)-S'_s(x))\1_{\{\left|S-S'\right|\leq n\}}\nu(dxds)=l\int_0^t\int_\cO sgn(S-S')\left|S_s(x)-S'_s(x)\right|^{l-1}\nu(dxds)\leq 0$$
and\begin{eqnarray*}\int_0^t\int_\cO\varphi_n^{\prime}(S_s(x)&-&S'_s(x))I_{\{\left|S-S'\right|>n\}}\nu(dxds)\\&=&\int_0^t\int_\cO
n^{l-2}[l(l-1)(\left|S-S'\right|-n)sgn(S-S')+sgn(S-S')ln]\nu(dxds)\leq0\end{eqnarray*}
By the uniform ellipticity of the operator $A$ we get
\[
\mathcal{E} \big(\varphi_n^{\prime} (u_s-S'_s), \, u_s-S'_s \big) \,
\geq \, \lambda  \, \int_{\cO} \varphi_n^{\prime \prime} (u_s-S'_s)
|\nabla (u_s-S'_s)|^2 \, dx.
\]
  Let $\epsilon >0$ be fixed. Using the Lipschitz condition on $\bar{f}$ and the properties of the
functions $(\varphi_n)_n$ we get
\begin{equation*} \begin{split}   & |\varphi_n^{\prime} (u_s-S'_s)| \, |\bar{f}(s,x,u_s-S'_s, \nabla u_s-S'_s)| \\&
 \leq \,   |\varphi_n^{\prime} (u_s-S'_s)| \, \big( |\bar{f}^0 (s,x)| + C \, (|u_s-S'_s| +
 |\nabla (u_s-S'_s)|)\, \big)\\
 & \leq  \, |\varphi_n^{\prime} (u_s-S'_s)||\bar{f}^0 (s,x)|+ |u_s-S'_s||\varphi_n^{\prime\prime}
 (u_s-S'_s)| \, (C |u_s-S'_s| + C
 |\nabla (u_s-S'_s)|)\, )\\
 &\leq l (\varphi_n (u_s-S'_s) + 1) \, |\bar{f}^0(s,x) | + C|u_s-S'_s|^2|\varphi_n^{\prime\prime}
 (u_s-S'_s)| +C|u_s-S'_s| |\nabla (u_s-S'_s)||\varphi_n^{\prime\prime}
 (u_s-S'_s)|\\
 & \leq l (\varphi_n (u_s-S'_s) + 1) \, |\bar{f}^0(s,x) | + (C +
 c_{\epsilon})\, |u_s-S'_s|^2
 \varphi_n^{\prime\prime}(u_s-S'_s)  + \, \epsilon \varphi_n^{\prime\prime}(u_s-S'_s)|\nabla u_s-S'_s|^2 .\\
 \end{split}
 \end{equation*}
 Now using Cauchy-Schwarz inequality and the Lipschitz condition on $\bar{g}$ we get
 \begin{equation*}
 \begin{split}
 & \,\sum_{i=1}^d \varphi_n^{\prime\prime}(u_s-S'_s) \partial_i (u_s-S'_s)
 \, \bar{g}_i (s,x,u_s-S'_s,\nabla (u_s-S'_s)) \\&\leq \, \varphi_n^{\prime\prime}(u_s-S'_s)
 \, |\nabla (u_s-S'_s)|\, \big( |\bar{g}^0 (s,x)| + C|u_s-S'_s| +
 \alpha  |\nabla (u_s-S'_s)|  \, \big) \,   \\
 &   \leq \,  \epsilon \, \varphi_n^{\prime\prime} (u_s-S'_s)
 |\nabla (u_s-S'_s)|^2 + 2c_{\epsilon} \varphi_n^{\prime\prime} (u_s-S'_s)\, \big( K^2+
 C^2|u_s-S'_s|^2\, \big)+ \alpha \, \varphi_n^{\prime\prime} (u_s-S'_s)|\nabla (u_s-S'_s)|^2  \\
 & \leq   l(l-1)c_{\epsilon}K^2 + 2c_{\epsilon}(K^2 + C^2 ) l (l-1)
 |\varphi_n (u_s-S'_s)|+
 (\alpha + \epsilon) \, \varphi_n^{\prime\prime} (u_s-S'_s)
 |\nabla (u_s-S'_s)|^2.   \\
 \end{split}
 \end{equation*}
  In the same way as before
 %and  using the obvious inequality :
 %$ (a + b)^2 \leq c'_{\epsilon} a^2 + (1 + \epsilon) b^2$ for each $ \epsilon > 0 $
\begin{equation*}
 \begin{split}&
\sum_{j=1}^{\infty}  \varphi_n^{\prime \prime } (u_s-S'_s )
\bar{h}_{j}^2(s, u_s-S'_s,\nabla (u_s-S'_s))\\& \leq \,
\varphi_n^{\prime \prime }  (u_s-S'_s)
 \, \big( \, c'_{\epsilon}(|\bar{h}^0 (s,x)|+ C|u_s-S'_s|)^2  \, + \,
 (1 + \epsilon)\beta^2 \,  |\nabla (u_s-S'_s)|^2  \, \big) \\
 &\leq \varphi_n^{\prime \prime }  (u_s-S'_s)
 \, \big(2c'_{\epsilon}K^2 \, +\, 2c'_{\epsilon}C^2 |u_s-S'_s |^2 + \,
 (1 + \epsilon)\beta^2 \,  |\nabla (u_s-S'_s)|^2  \, \big)\\
 & \leq \, 2c'_{\epsilon} l (l-1)K^2 +  2c'_{\epsilon}( K^2 +C^2 )l (l-1) \varphi_n(u_s -S'_s)
 + (1 + \epsilon) \, \beta^2\, \varphi_n^{\prime \prime }(u_s-S'_s)|\nabla (u_s-S'_s)|^2.
\end{split}
 \end{equation*}
 Thus taking the expectation, we deduce
\begin{equation}
\label{e6}
\begin{split}
&E \, \int_{\mathcal{O}}\varphi_n (u_t (x)-S'_t(x)) \, dx + (\lambda
- \frac{1}{2}(1+ \epsilon)\beta^2 - (\alpha + 2 \epsilon) \, ) \, E
\, \int_0^t \int_{\mathcal{O}} \varphi_n^{\prime \prime} (u_s -S'_s)
\,
|\nabla( u_s-S'_s) |^2  \, dx \, ds \\
& \, \leq \, l(l-1)c''_{\epsilon} K^2 \, + \,  c''_{\epsilon} l
(l-1)\big( K^2 +C^2+C+ c_{\epsilon}\big)
E \, \int_0^t \int_{\mathcal{O}} \varphi_n (u_s(x) -S'_s(x))\, dx \, ds. \\
\end{split}
\end{equation}
On account  of the contraction condition, one can choose $ \epsilon
> 0 $ small enough such that $$ \lambda - \frac{1}{2}(1+
\epsilon)\beta^2 - (\alpha + 2 \epsilon)>0 $$  and then
\begin{equation*}
\begin{split}
E\, \int_{\mathcal{O}} \varphi_n (u_t (x)-S'_t(x)) \, dx    &
 \leq \, c K^2 l (l-1) \, + \,  cl (l-1)
E \, \int_0^t \int_{\mathcal{O}} \varphi_n (u_s (x)-S'_s(x))\, dx \, ds \, .\\
\end{split}
\end{equation*}
 We obtain by Gronwall's Lemma, that
\begin{equation}\label{varphi1}
\begin{split}
E\, \int_{\mathcal{O}} \varphi_n (u_t (x)-S'_t(x)) \, dx    &
 \leq \, c\,K^2  l (l-1) \, \exp{ \big( c \, l (l-1)\,t \big) }\\
\end{split}
\end{equation}
 and  so it is now easy from (\ref{e6})  to get
\begin{equation}\label{varphi2}
\begin{split}
  E \, \int_0^t \int_{\mathcal{O}} \varphi_n^{\prime \prime} (u_s (x)-S'_s(x))  \,
|\nabla (u_s-S'_s) |^2  \, dx \, ds  \, \leq \,
 c^{\prime}\, K^2 l \, (l-1) \,  \exp{ \big( c l (l-1)\,t \big)}.\\
\end{split}
\end{equation}
 Finally, letting $ n \to \infty $  by Fatou's lemma we deduce (\ref{uS'l}) and
 (\ref{nablauS'l}).
\\Then with (\ref{uS'It\^o's}), we know that
$$-\int_0^t\int_\cO\varphi'_n(u_s-S'_s)\nu(dxds)=-\int_0^t\int_\cO\varphi'_n(S_s-S'_s)\nu(dxds)\leq C.$$
This yields (\ref{nuS'l}) by Fatou's lemma.
\end{proof}
With the help of Lemma \ref{estimateul}, we are able to prove the
following It\^o formula:
\begin{proposition}
\label{Lp} Assume the hypotheses of the previous lemma. Let
$(u,\nu)$ be the solution of the problem (\ref{SPDEO}). Then for $
l\geq2$, we get the following It\^o's formula, $ P$-almost surely,
for all $t\in [0,T]$,
\begin{equation}
\label{LpIt\^o's}
\begin{split}
&  \int_{\mathcal{O}}\left| u_t (x)-S'_t(x) \right|^l  \, dx  +
\int_0^t \cE \, \big(l\, (u_s-S'_s)^{l -1}sgn(u_s-S'_s) , \,
u_s-S'_s \big) \, ds =
\int_{\mathcal{O}} \left|\xi (x)-S'_0(x) \right|^l \, dx  \\
& +l\int_{0}^{t}\int_{\mathcal{O}}sgn(u_s-S'_s)
 \left|u_s (x)-S'_s(x) \right|^{l-1} \bar{f} (s,x, u_s-S'_s,\nabla (u_s-S'_s)) \, dx ds\\
 & -l (l-1) \, \sum_{i=1}^{d} \,\int_{0}^{t} \int_{\mathcal{O}} \left|u_s
(x)-S'_s(x)  \right|^{l -2}
\partial_i (u_{s}(x)-S'_s(x))\, \bar{g}_{i} (s, x,u_s-S'_s,\nabla (u_s-S'_s)) \, dx \, ds \\
& + l \, \sum_{j=1}^{\infty} \int_{0}^{t}
\int_{\mathcal{O}}sgn(u_s-S'_s)
 \,  \left|u_s (x)-S'_s(x) \right|^{l-1}\bar{h}_{j}( s,x, u_s-S'_s,\nabla (u_s-S'_s))\, dx dB_{s}^{j}\\
 & + \frac{l (l-1)}{2}\, \sum_{j=1}^{\infty}
\int_{0}^{t}\int_{\mathcal{O}}
 \left|u_s (x)-S'_s(x) \right|^{l-2} \bar{h}_{j}^2(s,x, u_s-S'_s,\nabla (u_s-S'_s)) \, dx\, ds\\&+l\int_{0}^{t}\int_{\mathcal{O}}sgn(u_s-S'_s)
 \left|u_s (x)-S'_s(x) \right|^{l-1} \nu(dx\, ds)\, .\\
\end{split}
\end{equation}
\end{proposition}
\begin{proof}
From It\^o's formula (see Theorem 5 in \cite{DMZ12}), with the same
notations as in the previous lemma, we have $ P$-almost surely, and
for all $t\in [0,T]$ and all $ n \in \bbN^*$,
\begin{equation*}
\begin{split}
& \int_{\mathcal{O}} \varphi_n (u_t (x)-S'_t(x)) \, dx \; + \;
\int_0^t \mathcal{E} \big(\varphi_n^{\prime} (u_s-S'_s), \, u_s-S'_s
\big) \, ds = \int_{\mathcal{O}} \varphi_n (\xi(x)-S'_0(x) ) \, dx
\\ & + \int_{0}^{t}\int_{\mathcal{O}}
\varphi_n^{\prime}(u_s(x)-S'_s(x)) \bar{f}  (s,x, u_s-S'_s,\nabla (u_s-S'_s)) \, dx ds \\
&  - \sum_{i=1}^{d} \int_{0}^{t} \int_{\mathcal{O}}
\varphi_n^{\prime \prime } (u_s(x)-S'_s(x) )
\partial_i (u_s(x)-S'_s(x)) \, \bar{g}_{i} (s,x,u_s-S'_s, \nabla (u_s-S'_s)) \, dx \, ds \\& +
\sum_{j=1}^{\infty} \int_{0}^{t} \int_{\mathcal{O}}
\varphi_n^{\prime} (u_s(x)-S'_s(x)) \,  \bar{h}_{j}( s,x, u_s-S'_s,\nabla (u_s-S'_s))\, dx dB_{s}^{j}\\
& + \frac{1}{2}\, \sum_{j=1}^{\infty} \int_{0}^{t}\int_{\mathcal{O}}
\varphi_n^{\prime \prime } (u_s(x)-S'_s(x)) \bar{h}_{j}^2(s,x,
u_s-S'_s,\nabla (u_s-S'_s)) \, dx\, ds\\&
+\int_0^t\int_\cO\varphi_n^{\prime}(u_s(x)-S'_s(x))\nu(dxds) \, .\\
\end{split}
\end{equation*}
%The last term is negative and the two terms in the LHS are positive, hence, the last term is controlled by the other terms in the RHS which are bounded.
Therefore, passing to the limit as $ n \to \infty $, the
convergences come from the  Lemma \ref{estimateul} and the dominated
convergence theorem.
\end{proof}

From now on, we assume  the following stronger
hypothesis:\begin{equation}\label{alphabet}\alpha+\frac{1}{2}\beta^2+72\beta^2<\lambda.\end{equation}
At this stage, the idea is to adapt the Moser iteration technics to
our setting. To this end, in order to control uniformly the $L^l$
norms and make $l$ tend to $+\infty$, we introduce for each
$l\geq2$,  the processes $v$ and $v'$ given by
\begin{equation*}
\begin{split}
\label{v} v_t :&=\sup _{s\leq t}\left( \int_{\cO}\left|
u_s-S'_s\right|^l dx +\gamma l\left( l-1\right)
\int_0^s\int_{\cO}\left| u_r-S'_r\right|
^{l-2}\left| \nabla( u_r-S'_r)\right| ^2\, dx\, dr\right),  \\
 v_t^{\prime } :&=\int_{\cO}\left| \xi-S'_0 \right| ^l\, dx+l^2 c_1
\left\| \left| u-S'\right| ^l\right\| _{1,1;t} +l\left\|
\bar{f}^0\right\| _{\theta ,t}^{*}\left\| \left| u-S'\right|
^{l-1}\right\| _{\theta ;t} \\&\quad+ l^2 \left( c_2\left\|
|\bar{g}^0 |^2\right\| _{\theta;t}^{*}
 + c_3\left\| |\bar{h}^0 |^2\right\| _{\theta ;t}^{*}\right) \left\|
\left| u-S'\right|^{l-2}\right\|_{\theta ;t},\\
\end{split}
\end{equation*}
where the constants are given by
\begin{equation}
\label{const}
\begin{split}
&   \gamma =\lambda -\alpha -\frac{\epsilon l}{l-1}-\frac{1+\epsilon }{2} \beta ^2\\
& c_1=\frac {C}{2} \left( 1+\frac{C}{4\epsilon }\right)
+\frac{3+2\epsilon }{ 2\epsilon}C^2+3\frac{1+\epsilon}{\epsilon^2}C^2\\
&  c_2=\frac 1{2\epsilon } \quad \mbox{and} \quad c_3=\frac{\left(
3+\epsilon
\right) \left( 1+\epsilon \right) }{\epsilon} \\
\end{split}
\end{equation}
The main difficulty in the stochastic case is to control the
martingale part. We start by estimating  the bracket of the local
martingale
 in \eqref{LpIt\^o's}$$ M_t := l \sum_{j=1}^{\infty}
\int_{0}^{t} \int_{\mathcal{O}}sgn(u_s-S'_s)
 \,\left| u_s (x)-S'_s(x)\right|^{l-1}\bar{h}_{j}( s,x, u_s-S'_s,\nabla (u_s-S'_s))\, dx
 dB_{s}^{j}$$
\begin{lemma}
\label{bracket} For arbitrary $\varepsilon >0,$ one has

\begin{equation}
\label{e5}
\begin{split}
\left\langle M\right\rangle _t^{\frac 12}& \le \varepsilon
v_t+\frac{l^2}{ 2\varepsilon }\left( \frac{1+\varepsilon}\varepsilon
\left\| |\bar{h}^0 |^2\right\| _{\theta ;t}^{*}\left\| \left|
u-S'\right| ^{l-2}\right\| _{\theta ;t}+\frac{ 1+\varepsilon
}\varepsilon C^2\left\| \left| u-S'\right| ^l\right\|
_{1,1;t}\right) \\
& + \sqrt{1+\varepsilon }\sqrt{\frac l{l-1}}\frac \beta {\sqrt{\gamma }%
}v_t \, .
\end{split}
\end{equation}
\end{lemma}
The proof is the same as Lemma 12 in \cite{DMS05} replacing $u$ by
$u-S'$ and also $h$ by $\bar{h}$.

In what follows we will use the notion of domination, which is
essential to handle the martingale part. We recall the definition
from  Revuz and Yor \cite{RevuzYor}.
\begin{definition}
A non-negative, adapted right continuous process $X$ is dominated by
an increasing process $A$, if
$$
E \, \big[X_{\rho} \,] \, \leq \, E \, \big[A_{\rho} \,]
$$
for any  bounded stopping time, $ \rho$.
\end{definition}
One important result related to this notion is the following
domination inequality (see Proposition IV.4.7 in Revuz-Yor, p. 163),
for any $ k\in ]0,\, 1[$,
\begin{equation}
\label{domination} E \big[ (X_\infty^*)^k \, \big] \, \leq \, C_k\,
E \big[ (A_\infty )^k \, \big]
\end{equation}
where $ C_k$ is a positive constant and $ X_t^* := \sup_{ s \leq
t} |X_s|$. \\[0.2cm]
We will  also use the fact that if $A,A^{\prime }$ are increasing
processes, then the
domination of a process $X$ by $A$ is equivalent to the domination of $%
X+A^{\prime }$ by $A+A^{\prime }.$

\begin{lemma}
\label{tau}
 The Process $ \tau v$ is dominated by the process $
v'$ where
\[ \tau =1-6\epsilon -6\sqrt{1+\epsilon }\sqrt{\frac l{l-1}}\frac{\beta} {\sqrt{ \gamma }}.\]
In other words, we have
\begin{equation}
\begin{split}
& \tau \, E \sup _{0 \leq s\leq t}\left( \int_{\cO}\left|
u_s-S'_s\right|^l\, dx +\gamma l\left( l-1\right)
\int_0^s\int_{\cO}\left| u_r-S'_r\right| ^{l-2}\left| \nabla
(u_r-S'_r)\right| ^2dx dr\right) \\& \leq  E \int_{\cO}\left|
\xi-S'_0 \right|^l dx  \, +\, l^2c_1 E \, \left\| \left| u-S'\right|
^l\right\| _{1,1;t} +l E \left\| \bar{f}^0\right\| _{\theta
,t}^{*}\left\| \left| u-S'\right| ^{l-1}\right\| _{\theta
;t}\\&\quad+ \, l^2 E \, \left( c_2 \left\| |\bar{g}^0 |^2\right\|
_{\theta ;t}^{*}+c_3\left\| |\bar{h}^0 |^2 \right\| _{\theta
;t}^{*}\right) \left\| \left| u-S'\right|^{l-2}\right\|_{\theta ;t},
\end{split}
\end{equation}
where $  \gamma, \; c_1, \, c_2  $ and $ c_3$ are the constants
given above.
\end{lemma}
\begin{proof}
Starting from the relation (\ref{LpIt\^o's}):
\begin{equation*}
\begin{split}
&  \int_{\mathcal{O}}\left| u_t (x)-S'_t(x) \right|^l  \, dx  +
\int_0^t \cE \, \big(l\, (u_s-S'_s)^{l -1}sgn(u_s-S'_s) , \,
u_s-S'_s \big) \, ds =
\int_{\mathcal{O}} \left|\xi (x)-S'_0(x) \right|^l \, dx  \\
& +l\int_{0}^{t}\int_{\mathcal{O}}sgn(u_s-S'_s)
 \left|u_s (x)-S'_s(x) \right|^{l-1} \bar{f} (s,x, u_s-S'_s,\nabla (u_s-S'_s)) \, dx ds\\
 & -l (l-1) \, \sum_{i=1}^{d} \,\int_{0}^{t} \int_{\mathcal{O}} \left|u_s
(x)-S'_s(x)  \right|^{l -2}
\partial_i (u_{s}(x)-S'_s(x))\, \bar{g}_{i} (s, x,u_s-S'_s,\nabla (u_s-S'_s)) \, dx \, ds \\
& + l \, \sum_{j=1}^{\infty} \int_{0}^{t}
\int_{\mathcal{O}}sgn(u_s-S'_s)
 \,  \left|u_s (x)-S'_s(x) \right|^{l-1}\bar{h}_{j}( s,x, u_s-S'_s,\nabla (u_s-S'_s))\, dx dB_{s}^{j}\\
 & + \frac{l (l-1)}{2}\, \sum_{j=1}^{\infty}
\int_{0}^{t}\int_{\mathcal{O}}
 \left|u_s (x)-S'_s(x) \right|^{l-2} \bar{h}_{j}^2(s,x, u_s-S'_s,\nabla (u_s-S'_s)) \, dx\, ds\\&+l\int_{0}^{t}\int_{\mathcal{O}}sgn(u_s-S'_s)
 \left|u_s (x)-S'_s(x) \right|^{l-1} \nu(dx\, ds)\, ,\ \ \ a.s.\\
\end{split}
\end{equation*}
The last term is negative: from the condition of minimality, we have
the following relation,
\begin{eqnarray*}&&\int_{0}^{t}\int_{\mathcal{O}}sgn(u_s-S'_s)
 \left|u_s (x)-S'_s(x) \right|^{l-1} \nu(dxds)\\&=&\int_{0}^{t}\int_{\mathcal{O}}sgn(S_s-S'_s)
 \left|S_s (x)-S'_s(x) \right|^{l-1} \nu(dxds)\leq0.\end{eqnarray*}
Then we can do the same calculus as in the proof of Lemma 14 in
\cite{DMS05}, replacing $u$ by $u-S'$ and $f$, $g$, $h$ by
$\bar{f}$, $\bar{g}$, $\bar{h}$ respectively. \end{proof} The proofs
of the next 3 lemmas are similar to the proofs of Lemmas 15, 16 and
17 in \cite{DMS05}, just replacing $u$ by $u-S'$ and replacing $f$,
$g$ and $h$ by $\bar{f}$, $\bar{g}$ and $\bar{h}$ respectively.

\begin{lemma}
\label{sob}

The process $v$ satisfies the estimate%
$$
v_t\ge \delta \left\| \left| u-S'\right| ^l\right\| _{0;t}
$$
with $\delta = 1 \wedge \left(2c_S^{-1} \gamma\right) ,$ where $c_S$
is the constant in the Sobolev inequality (\ref{Sobolev}).
\end{lemma}

\begin{lemma} \label{sup}
The process
$$
w_t :=\left[ \left\| \left| u-S'\right| ^{\sigma l}\right\| _{\theta
;t}^{\frac 1\sigma }\vee \left\| \xi-S'_0\right\| _\infty ^l\vee
\left\| \bar{f}^0\right\| _{\theta ;t}^{l} \vee \left\| |\bar{g}^0
|^2\right\| _{\theta ;t}^{*\frac{l}{2}}\vee \left\| |\bar{h}^0 |^2
\right\| _{\theta ;t}^{*\frac{l}{2}}\right]
$$
is dominated by the process
$$
w_t^{\prime }:=6k\left( t\right) l^2\left[ \left\| \left|
u-S'\right| ^{l}\right\| _{\theta ;t}\vee \left\| \xi -S'_0\right\|
_\infty ^l\vee \left\| \bar{f}^0\right\| _{\theta ;t}^{l}\vee
\left\| |\bar{g}^0 |^2 \right\| _{\theta ;t}^{*\frac{l}{2}}\vee
\left\| |\bar{h}^0 |^2 \right\| _{\theta ;t}^{*\frac{l}{2}}\right] ,
$$
where $ \sigma = \frac{d + 2\theta}{d}$ and $k:\bbR_{+}\rightarrow
\bbR_{+}$ is a function independent of $l$, depending only on the
structure constants.
\end{lemma}

\begin{lemma}{\label{estimev}}
There exists a function $k_1:\bbR_{+}\times \bbR_{+}\rightarrow
\bbR_{+}$ which involves only the structure constants of our problem
and such that the
following estimate holds%
$$
Ev_t\le k_1\left( l,t\right) E\left( \int_{\cO}\left| \xi-S'_0
\right| ^l dx +\left\| \bar{f}^0\right\| _{\theta ;t}^{*l}+\left\|
|\bar{g}^0 |^2 \right\| _{\theta ;t}^{*\frac{l}{2}}+\left\|
|\bar{h}^0 |^2 \right\| _{\theta ;t}^{*\frac{l}{2}}\right) .
$$
\end{lemma}

\vspace{0.2cm}

We now prove Theorem \ref{LPESTIM} in the case where $\xi$, $\bar{f}^0$, $\bar{g}^0$ and $\bar{h}^0$ are uniformly bounded:\\

We set $l=p\sigma ^n,$ with some $n\in \bbN^*.$ By Lemma \ref{sup}
and the domination inequality (\ref{domination})
we deduce, for $n\ge 1,$%
\begin{equation*}
\begin{split}
& E\left( \left\| \left| u-S'\right| ^{\sigma l}\right\| _{\theta
;t}^{\frac 1\sigma }\vee \left\| \xi-S'_0 \right\| _\infty ^l\vee
\left\| \bar{f}^0\right\| _{\theta ;t}^{*l}\vee \left\| |\bar{g}^0
|^2\right\| _{\theta ;t}^{*\frac{l}{2}}\vee \left\| |\bar{h}^0
|^2\right\| _{\theta
;t}^{*\frac{l}{2}}\right) ^{\frac 1{\sigma ^n}}\\
& \leq \,  C_{\sigma ^{-n}}\left( 6k\left( t\right) l^2\right)
^{\frac 1{\sigma ^n}}E\left( \left\| \left| u-S'\right| ^l\right\|
_{\theta ;t}\vee \left\| \xi-S'_0 \right\| _\infty ^l\vee \left\|
\bar{f}^0\right\| _{\theta ;t}^{*l}\vee \left\| |\bar{g}^0
|^2\right\| _{\theta ;t}^{*\frac{l}{2}}\vee \left\| |\bar{h}^0
|^2\right\| _{\theta ;t}^{*\frac{l}{2}}\right) ^{\frac 1{\sigma
^n}},
\end{split}
\end{equation*}
where $C_{\sigma ^{-n}}$ is the constant in the domination
inequality. This constant is estimated by
$$
C_{\sigma ^{-n}}\le \sigma ^{\frac n{\sigma ^n}}\left( 1-\frac
1{\sigma ^n}\right) ^{-1}.
$$
(See the exercise IV.4.30 in Revuz -Yor, p. 171). So let us denote
by
$$
a_n :=\left\| \left| u-S'\right| ^{p\sigma ^n}\right\| _{\theta
;t}^{\frac 1{\sigma ^n}}\vee \left\| \xi-S'_0 \right\| _\infty
^p\vee \left\| \bar{f}^0\right\| _{\theta ;t}^{*p}\vee \left\|
|\bar{g}^0|^2\right\| _{\theta ;t}^{*\frac{p}{2}}\vee \left\|
|\bar{h}^0|^2\right\| _{\theta ;t}^{*\frac{p}{2}}
$$
and deduce from the above inequality the following one%
$$
Ea_{n+1}\le \sigma ^{\frac n{\sigma ^n}}\left( 1-\frac 1{\sigma
^n}\right) ^{-1}\left( 6k\left( t\right) \left( p\sigma ^n\right)
^2\right) ^{\frac 1{\sigma ^n}}Ea_n.
$$
Iterating this relation $n$ times we get%
$$
Ea_{n+1}\le \sigma ^{3\sum_{m=1}^n\frac m{\sigma
^m}}\prod_{m=1}^n\left( 1-\frac 1{\sigma ^m}\right) ^{-1}\left(
6k\left( t\right) p^2\right) ^{\sum_{m=1}^n\frac 1{\sigma ^m}}Ea_1.
$$
Now we shall let $n$ tend to infinity in this relation. Since in
general one
has%
$$
\lim _{q,q^{\prime }\rightarrow \infty }\left\| F\right\|
_{q,q^{\prime };t}=\left\| F\right\| _{\infty ,\infty ;t},
$$
for any function $F:\bbR_{+}\times \cO \rightarrow \bbR,$ it is easy
to see that $ \lim _{n \to \infty} \left\| \left| u-S'\right|
^{p\sigma ^n}\right\| _{\theta ;t}^{\frac 1{\sigma ^n}}=\left\|
u-S'\right\| _{\infty ,\infty ;t}^p.\,$\\Therefore we have
$$
\lim_{n \to \infty} a_n=\left\| u-S'\right\| _{\infty ,\infty
;t}^p\vee \left\| \xi-S'_0\right\| _\infty ^p\vee \left\|
\bar{f}^0\right\| _{\theta ;t}^{*p}\vee \left\| |\bar{g}^0 |^2
\right\| _{\theta ;t}^{*\frac{p}{2}}\vee \left\| |\bar{h}^0
|^2\right\| _{\theta ;t}^{*\frac{p}{2}},
$$
which implies
$$
E\left\| u-S'\right\| _{\infty ,\infty ;t}^p\le \rho \left( t\right)
Ea_1,
$$
with
$$
\rho \left( t\right) =\sigma ^{3\sum_{m=1}^\infty \frac m{\sigma
^m}}\prod_{m=1}^\infty \left( 1-\frac 1{\sigma ^m}\right)
^{-1}\left( 5k\left( t\right) p^2\right) ^{\sum_{m=1}^\infty \frac
1{\sigma ^m}}.
$$
Now we estimate $Ea_1$ by using the fact that $\delta \left\| \left|
u-S'\right| ^{p\sigma }\right\| _{\theta ;t}^{\frac 1\sigma }\le
v_t,$ with $p$
replacing $l$ in the expression of $v.$ So we have%
\begin{equation*}
\begin{split} Ea_1& =E\left( \left\| \left| u-S'\right| ^{p\sigma }\right\|
_{\theta ;t}^{\frac 1\sigma }\vee \left\| \xi-S'_0\right\| _\infty
^p\vee \left\| \bar{f}^0\right\| _{\theta ;t}^{*p}\vee \left\|
|\bar{g}^0 |^2 \right\| _{\theta ;t}^{*\frac{p}{2}}\vee \left\|
|\bar{h}^0 |^2\right\|
_{\theta ;t}^{*\frac{p}{2}}\right)\\
 &\leq E\left( \delta ^{-1}v_t+\left\| \xi-S'_0 \right\|
_\infty ^p+\left\| \bar{f}^0\right\| _{\theta ;t}^{*p}+\left\|
|\bar{g}^0 |^2 \right\| _{\theta ;t}^{*\frac{p}{2}}\vee \left\|
|\bar{h}^0 |^2\right\| _{\theta ;t}^{*\frac{p}{2}}\right) .
\end{split}
\end{equation*}
 Finally one deduces the following estimate by applying Lemma \ref{estimev}  with $l=p$:
\begin{equation}\label{Maxestimate}
 E\left\| u-S'\right\| _{\infty ,\infty ;t}^p\le k_2\left(
t\right) E\left( \left\| \xi-S'_0\right\| _\infty ^p+\left\|
\bar{f}^0\right\| _{\theta ,t}^{*p}+\left\| |\bar{g}^0 |^2 \right\|
_{\theta ;t}^{*p/2}+\left\| |\bar{h}^0 |^2\right\| _{\theta
;t}^{*p/2}\right).
\end{equation}

Moreover (see Theorem 11 \cite{DMS05}), we have
\begin{eqnarray*}E\left\|S'\right\|_{\infty,\infty;t}^p\leq
k(t)E\left(\left\|S'_0\right\|_\infty^p+\left\|f'\right\|_{\theta;t}^{*p}+\left\||g'|^2\right\|_{\theta;t}^{*p/2}+\left\||h'|^2\right\|_{\theta;t}^{*p/2}\right)\end{eqnarray*}

Hence, \begin{equation*}\begin{split}E\left\|
u\right\|^p_{\infty,\infty;t}&\leq
c(p)(E\left\|u-S'\right\|_{\infty,\infty;t}^p+E\left\|S'\right\|_{\infty,\infty;t}^p)\\&\leq
c(p)k(t)E\big(\left\|\xi\right\|_\infty^p+\left\|S'_0\right\|_\infty^p+\left\|f'\right\|_{\theta;t}^{*p}+\left\||g'|^2\right\|_{\theta;t}^{*p/2}+\left\||h'|^2\right\|_{\theta;t}^{*p/2}\\&\quad+\left\|\bar{f}^0\right\|_{\theta;t}^{*p}+\left\||\bar{g}^0|^2\right\|_{\theta;t}^{*p/2}+\left\||\bar{h}^0|^2\right\|_{\theta;t}^{*p/2}\big).\end{split}\end{equation*}
This ends the proof of Theorem \ref{LPESTIM} in this particular case
where $\xi$, $\bar{f}^0$, $\bar{g}^0$ and $\bar{h}^0$ are uniformly
bounded. We now turn out to the general case.

\subsection{Proof of Theorem  \ref{LPESTIM}  in the general case} \label{casgen}
We now assume that {\bf (H)}, {\bf (O)}, {\bf (HI$\mathbf{2p}$)},
{\bf (HO$\mathbf{\infty p}$)} and  {\bf (HD$\mathbf{\theta p}$)}
hold. We are going to prove Theorem \ref{LPESTIM} in the general
case by using an approximation argument. For this, for all
$n\in\Ne$, $1\leq i\leq d ,1\leq j\leq\infty$ and all $(t,w,x,y,z)$
in $\R^+ \times\Omega\times\cO\times\R \times\R^d$, we set
\begin{eqnarray}\label{approximation}
 \bar{f}_n (t,w,x,y,z)&=& \bar{f}(t,w,x,y,z) -\bar{f}^0 (t,w,x) +\bar{f}^0 (t,w,x)\cdot{\bf 1}_{\{ |\bar{f}^0 (t,w,x)|\leq n\}}\nonumber\\
 \bar{g}_{i,n}(t,w,x,y,z)&=& \bar{g}_i (t,w,x,y,z) -\bar{g}_i^0 (t,w,x) +\bar{g}_i^0 (t,w,x)\cdot{\bf 1}_{\{ |\bar{g}_i^0 (t,w,x)|\leq n\}}\nonumber\\
\bar{h}_{j,n}(t,w,x,y,z)&=& \bar{h}_j (t,w,x,y,z) -\bar{h}_j^0
(t,w,x) +\bar{h}_j^0 (t,w,x)\cdot{\bf 1}_{\{ |\bar{h}_j^0
(t,w,x)|\leq n\}}\nonumber\\\xi_n (w,x)&=& \xi (w,x)\cdot {\bf
1}_{\{ |\xi(\omega, x)|\leq n\}}
\end{eqnarray}
One can check that for all $n$, $\bar{f}_n$, $\bar{g}_n$,
$\bar{h}_n$ and $\xi^n-S'_0$ satisfy all the assumptions of the Step
1 of the proof, and that Lipschitz constants do not depend on $n$.
And the obstacle $S-S'$ is controlled by 0, which  obviously
satisfies {\bf(HO2)}. For each $n\in\Ne$, we put
$(\bar{u}^n,\nu^n)=\cR(\xi^n-S'_0,\bar{f}^n,\bar{g}^n,\bar{h}^n,S-S')$
and we know that $\bar{u}^n$ satisfies the estimate of Step 1. We
are now going to prove that $(\bar{u}^n,\, \nu^n)$ converges to
$(\bar{u},\nu)=\cR(\xi-S'_0,\bar{f},\bar{g},\bar{h},S-S')$.
\\Let us fix $n\leq m$ in $\Ne$ and put $\bar{u}^{n,m} :=\bar{u}^n -\bar{u}^m$ and $\nu^{n,m}:=\nu^n-\nu^m$
We first note that $\bar{u}^{n,m}$ satisfies the equation%
\begin{equation*}
\begin{split}
 d\bar{u}_t^{n,m}\left( x\right) +A\bar{u}_t^{n,m}\left( x\right)
dt &=\bar{f}_{n,m}\left( t,x,\bar{u}_t^{n,m}\left( x\right) ,\nabla
\bar{u}_t^{n,m}\left( x\right) \right) dt \\
&\quad   -\sum_{i=1}^d\partial _i\bar{g}_{i,n,m}\left(
t,x,\bar{u}_t^{n,m}\left( x\right) ,\nabla \bar{u}_t^{n,m}\left(
x\right) \right) dt\\& \quad
+\sum_{j=1}^{\infty}\bar{h}_{j,n,m}\left( t,x,\bar{u}_t^{n,m}\left(
x\right) ,\nabla \bar{u}_t^{n,m}\left( x\right) \right) dB_t^j +
\nu^{n,m}(x,\, dt)
\end{split}
\end{equation*}

where
\begin{equation*}
\begin{split}
\bar{f}_{n,m}\left( t,w,x,y,z\right)  & =\bar{f}\left(
t,w,x,y+\bar{u}_t^m\left( x\right) ,z+\nabla \bar{u}_t^m\left(
x\right) \right)   -\bar{f}\left(
t,w,x,\bar{u}_t^m\left( x\right) ,\nabla \bar{u}_t^m\left( x\right) \right) \\
& \quad +\bar{f}_n^0\left( t,w,x\right) -\bar{f}_m^0\left(
t,w,x\right)
\end{split}
\end{equation*}
and $\bar{g}_{i,n,m},\bar{h}_{j,n,m}$ have similar expressions.
Clearly one has
$$
\bar{f}_{n,m}\left( t,w,x,0,0\right) =\bar{f}_n^0\left( t,w,x\right)
-\bar{f}_m^0\left( t,w,x\right) :=\bar{f}_{n,m}^0\left( t,w,x\right)
$$
and some similar relations for $\bar{g}_{i,n,m}\left(
t,w,x,0,0\right) $ and $ \bar{h}_{j,n,m}\left( t,w,x,0,0\right) .$
On the other hand, one can easily verify that

\begin{equation*}
\begin{split}
& E\left\| \xi _n-\xi \right\| _\infty ^p\longrightarrow 0,\quad
\quad \quad  E\left\|
\bar{f}_n^0-\bar{f}^0\right\| _{\theta;T}^{*p}\longrightarrow 0, \\
& E\left\|\bar{g}_n^0-\bar{g}^0\right\| _{\theta ;T}^{*p}
\longrightarrow 0,\qquad E\left\|\bar{h}_n^0-\bar{h}^0\right\|
_{\theta;T}^{*p}\longrightarrow 0.
\end{split}
\end{equation*}

%Because that $(u^{n,m},\nu^{n,m})$ is not necessarily a solution of the problem \ref{SPDEO}, we can't use Lemma \ref{estimev} directly. We do as %follows: Applying It\^o's formula to $(u^{n,m})^2$, we get
%\begin{eqnarray*}&&\left\|u_t^{n,m}\right\|^2+2\int_0^t\mathcal{E}(u^{n,m}_s)ds=\left\|\xi^n-\xi^m\right\|^2+2\int_0^t(u_s^{n,m},\, f_{n,m}(s,x,u^{n,m},       %\nabla u^{n,m}))ds\\&&-2\int_0^t(\nabla u^{n,m}_s,\, g_{n,m}(s,x,u^{n,m},\nabla u^{n,m}))ds+2\int_0^t(u^{n,m}_s,h_{n,m}(s,x,u^{n,m},\nabla u^{n,m}))     %\\&&+\int_0^t\left\|h_{n,m}(s,x,u^{n,m},\nabla u^{n,m})\right\|^2ds+2\int_0^t\int_\cO u^{n,m}_s(x)\nu^{n,m}(dx\, ds)
%\end{eqnarray*}It is clear that the last term in the right hand side is negative. And also we can easily check that $f_{n,m}$, $g_{n,m}$ and $h_{n,m}$ %satisfy the Lipschitz condition. A similar calculation as before!!!

By Lemma \ref{estimvdiff} with $l=2$  (see Appendix) we deduce that
\begin{equation}\label{unum}
E\left\| \bar{u}^n-\bar{u}^m\right\| _T^2 \longrightarrow 0,\quad
as\ n,\, m\rightarrow \infty .
\end{equation}
Therefore, $(\bar{u}^n)$ has a limit $\bar{u}$ in $\cH_T$.

We now study the convergence of $(\nu^n )$.  Denote by $v^n$ the
parabolic potential associated to $\nu^n$,  and $z^n=\bar{u}^n-v^n$,
so $z^n$ satisfies the following SPDE\begin{equation*}
\begin{split}
dz_t^n (x)  +
 Az_t^n (x)dt &=\bar{f}_n (t,x,\bar{u}_t^n (x),\nabla \bar{u}_t^n (x))dt-\sum_{i=1}^d
  \partial_i \bar{g}_{i,n}(t,x,\bar{u}_t^n(x),\nabla \bar{u}_t^n (x)
 )dt\\
 & +  \sum_{j=1}^{\infty} \bar{h}_{j,n}(t,x,\bar{u}_t^n(x),\nabla \bar{u}_t^n(x))\, dB_t^j .\\
\end{split}
\end{equation*}
We define $z^{1,n}$ to be the solution of the following SPDE with
initial value $\xi^n-S'_0$ and zero boundary condition:
\begin{equation*}
\begin{split}
dz_t^{1,n}(x) +Az_t^{1,n}(x)dt &=(\bar{f}_t (x,\bar{u}_t^n(x),\nabla
\bar{u}_t^n(x))-\bar{f}_t^0(x))dt-\sum_{i=1}^d \partial_i
(\bar{g}_{i,t}(x,\bar{u}_t^n(x),\nabla \bar{u}_t^n(x))\\&\ \ \ \
-\bar{g}^0_t(x))dt
+ \sum_{j=1}^{\infty} (\bar{h}_{j,t}(x,\bar{u}_t^n(x),\nabla \bar{u}_t^n(x))-\bar{h}_t^0(x))\, dB_t^j .\\
\end{split}
\end{equation*}
This is a linear SPDE in $z^{1,n}$, its solution uniquely exists and
belongs to $\cH_T$. Applying It\^o's formula to $(z^{1,n})^2$ and
doing a classical calculation, we get:
\begin{equation*}E\left\|z^{1,n}-z^{1,m}\right\|^2_T\leq CE\big(\left\|\xi^n-\xi^m\right\|_2^2+\left\|\bar{u}^n-\bar{u}^m\right\|^2_T\big)\rightarrow0,\quad as\ n,\ m\rightarrow\infty.\end{equation*}
Then, we define $z^{2,n}$ to be the solution of the following SPDE
with initial value $0$ and zero boundary condition:
\begin{equation*}
dz_t^{2,n}(x) +Az_t^{2,n}(x)dt =\bar{f}_n^0(t,x)dt-\sum_{i=1}^d
\partial_i \bar{g}_{i,n}^0(t,x))dt +  \sum_{j=1}^{\infty}
\bar{h}_{j,n}^0(x)\, dB_t^j .
\end{equation*}
This is still a linear SPDE in $z^{2,n}$, its solution uniquely
exists and from the proof of Theorem 11 in \cite{DMS05}, we know
that
$$E\left\|z^{2,n}-z^{2,m}\right\|^2_T\leq CE\left(\left\|\bar{f}_{n,m}^0\right\|_{\theta;T}^{*2}+\left\|\left|\bar{g}_{n,m}^0\right|^2\right\|_{\theta;T}^{*}+\left\|\left|\bar{h}_{n,m}^0\right|^2\right\|_{\theta;T}^{*}\right) \rightarrow0, \quad as\ n,\ m\rightarrow\infty.$$
%From the fact that $u^n$ converges strongly to $u$ in $\cH$ and $f$ is Lipschitz, we deduce that $f(u^n,\nabla u^n)-f^0$ converges strongly to
% $f(u,\nabla u)-f^0$ with respect to the norm $\left\|\cdot\right\|_{2,2;T}$, precisely,
%\begin{eqnarray*}\left\|f(u^n,\nabla u^n)-f^0-f(u,\nabla u)+f^0\right\|_{2,2;T}\leq C\left\| u^n-u\right\|_T\rightarrow0,\quad a.s.\end{eqnarray*}
%and the same case for $g(u^n,\nabla u^n)-g^0$ and $h(u^n,\nabla u^n)-h^0$.
This yields: $$ E\left\| z^n-z^m\right\| _T^2 \longrightarrow
0,\quad  as \quad n,\, m\rightarrow \infty .
$$
Hence, using (\ref{unum}) and the fact that $\bar{u}^n =z^n+v^n$, we
get:  $$ E\left\| v^n-v^m\right\| _T^2 \longrightarrow 0, \quad as\
n,\, m\rightarrow \infty .
$$
Therefore, $(v^n)$ has a limit $v$ in $\cH_T$.  So, by extracting a
subsequence, we can assume that $( v^n)$ converges to $v$ in
$\mathcal K$ almost-surely. Then, it's clear that $ v \in \mathcal
P$, and we denote by $\nu$ the random regular measure associated to
the potential $v$. Moreover, we have $P$-a.s.,
$\forall\varphi\in\cW_t^+$:
\begin{eqnarray*}\int_0^t\int_\cO\varphi(x,s)\nu(dxds)&=&\lim_{n\rightarrow\infty}\int_0^t\int_\cO\varphi(x,s)\nu^n(dxds)\\&=&\lim_{n\rightarrow\infty}\int_0^t-(v^n_s,\frac{\partial\varphi_s}{\partial s})ds+\int_0^t\cE(v^n_s,\varphi_s)ds\\&=&\int_0^t-(v_s,\frac{\partial\varphi_s}{\partial s})ds+\int_0^t\cE(v_s,\varphi_s)ds.\end{eqnarray*}

%On the other hand, the first step of the proof ensures the
%validity of the relation (\ref{Maxestimate}) for each $u^n-S'$ but not
%$u^{n,m}$ with $n\le m,$ $n,m\in \bbN^*,$ because that $u^{n,m}$ is not necessarily a solution of the obstacle problem.
As a consequence of Lemma \ref{Lemme63}  in the Appendix, we know
that
$$
E\left\| \bar{u}^n-\bar{u}^m\right\| _{\infty ,\infty ;T}^p
\longrightarrow 0.
$$
Therefore, we can apply  Proposition \ref{Lp} to $\bar{u}^n$ and
pass to the limit  and so we obtain that this proposition remains
valid in this case. Then, one can  end the proof by repeating the
first part of Step 1 starting from Proposition \ref{Lp}.\\We
conclude thanks to the uniqueness of the solution of the obstacle
problem ensuring that $\bar{u}$ is equal to $u-S'$. {\flushright
$\hfill \Box$}

\section{Maximum Principle for local solutions}
We now introduce the lateral condition on the boundary that we
consider:
\begin{definition}
If $u$ belongs to $\cH_{loc}$, we say that $u$ is non-negative on
the boundary of $\cO$ if $u^+$ belongs to $\cH_T$ and we denote it
simply: $u\leq0$ on $\partial\cO$. More generally, if $M$ is a
random field defined on $[0,T]\times \cO$, we note $u\leq M$ on
$\partial \cO$ if $u-M\leq 0$ on $\partial\cO$.
\end{definition}
\subsection{It\^o's formula for the positive part of a local solution}
%A solution obtained in a larger domain $\cD$ with null boundary
%conditions, when regarded on $\cO$ becomes a local solution. The
%minimality comes from
%$$\int\int_\cO(u-S)\nu(dxds)\leq\int\int_\cD(u-S)\nu(dxds)=0.$$
The following proposition represents a key technical result which
leads to a generalization of the estimates of the positive part of a
local solution. Let $(u,\nu)\in\mathcal{R}_{loc}(\xi,f,g,h,S)$,
denote by $u^+$ its
positive part.  For this we need the following notation:%
\begin{equation}
\label{f0+} \begin{split} &f^{u,0}=1_{\left\{ u>0\right\}
}f^0,\;g^{u,0}=1_{\left\{ u>0\right\} }g^0,\;h^{u,0}=1_{\left\{
u>0\right\} }h^0,\\
& f^{u,0+}=1_{\left\{ u>0\right\} }\left( f^0\vee 0\right) ,\;\xi
^{+}=\xi \vee 0.
\end{split}
\end{equation}

\begin{proposition}\label{It\^o'slocalposi}
Assume that $\partial\cO$ is Lipschitz and that $u^+$ belongs to
$\cH_T$, i.e. $u$ is non-positive on the boundary of $\cO$ and
that the data satisfy the following integrability conditions%
$$ E\left\| \xi ^{+}\right\| _2^2<\infty ,\;E\left( \left\|
f^{u,0}\right\| ^*_{\theta;t}\right) ^2<\infty ,\;E\left\|
g^{u,0}\right\| _{2,2;t}^2<\infty ,\;E\left\| h^{u,0}\right\|
_{2,2;t}^2<\infty , $$ for each $t\ge 0.$\\Let $\varphi
:{\mathbb{R}}\rightarrow {\mathbb{R}}$ be a function of class ${\cal
C}^2,$ which admits a bounded second order derivative and such that
$\varphi ^{\prime }\left( 0\right) =0.$ Then the following relation
holds, a.s., for each $t\in [0,T],$
\begin{eqnarray}\label{It\^o'spositivepart}&&\int_\mathcal{O}\varphi(u_t^+(x))dx+\int_0^t\mathcal{E}(\varphi'(u_s^+),u_s^+)ds=\int_\mathcal{O}\varphi(\xi^+(x))dx+\int_0^t\int_\cO\varphi'(u_s^+(x))f_s(x)dxds\nonumber\\&&-\sum_{i=1}^d\int_0^t\int_\cO\varphi''(u_s^+(x))\partial_iu_s^+(x)g_s^i(x)dxds+\frac{1}{2}\int_0^t\int_\cO\varphi''(u_s^+(x))\1_{\{u_s>0\}}|h_s(x)|^2dxds\nonumber\\&&+\sum_{i=1}^\infty\int_0^t\int_\cO\varphi'(u_s^+(x))h_s^j(x)dxdB_s^j+\int_0^t\int_\mathcal{O}\varphi'(u_s^+(x))\nu(dxds).\end{eqnarray}
\end{proposition}
\begin{proof}We consider  $\phi\in\mathcal{C}_c^\infty(\cO),\ 0\leq\phi\leq1$, and put $$\forall t\in[0,T],\quad w_t=\phi u_t.$$

%From the definition of weak solution, we have the following relation:
 %\begin{eqnarray*}&&(u_t,\varphi_t)-(\xi,\varphi_0)-\int_0^t(u_s,\partial_s\varphi_s)ds+\int_0^t\mathcal{E}(u_s,\varphi_s)ds+\sum_{i=1}^d\int_0^t(g^i_s(u_s,\nabla u_s),\partial_i\varphi_s)ds\nonumber\\
 %   &&=\int_0^t(f_s(u_s,\nabla u_s),\varphi_s)ds+\sum_{j=1}^{+\infty}\int_0^t(h^j_s(u_s,\nabla u_s),\varphi_s)dB^j_s+\int_0^t\int_{\mathcal{O}}\varphi_s(x)\nu(dx,ds);\end{eqnarray*}
%We replace $\varphi$ by $\phi\varphi$, note that $\phi$ does not depend on $t$, a direct calculation  yields:
% \begin{eqnarray*}&&(\phi u_t,\varphi_t)-(\phi\xi,\varphi_0)-\int_0^t(\phi u_s,\partial_s\varphi_s)ds+\int_0^t\mathcal{E}(\phi u_s,\varphi_s)ds+\sum_{i=1}^d\int_0^t(\phi g^i_s(u_s,\nabla u_s),\partial_i\varphi_s)ds\nonumber\\
  %  &&=\int_0^t(\phi f_s(u_s,\nabla u_s),\varphi_s)ds+\sum_{j=1}^{+\infty}\int_0^t(\phi h^j_s(u_s,\nabla u_s),\varphi_s)dB^j_s+\int_0^t\int_{\mathcal{O}}\phi\varphi_s(x)\nu(dx,ds)\\&&+\int_0^t(u_s\sum %a_{i,j}\partial_i\phi,\partial_j\varphi_s)ds-\int_0^t(\sum a_{i,j}(\partial_i\phi)(\partial_j u_s)+\sum(\partial_i\phi)g_{i,t},\varphi_s)ds;\end{eqnarray*}

A direct calculation yields the following relation:
$$dw_t=Lw_tdt+\bar{f}_tdt+\sum_{i=1}^d\partial_i\widetilde{g_{i,t}}dt+\sum_{j=1}^\infty\widetilde{h_{j,t}}dB_t^j+\phi\nu(x,dt)$$
where $$\bar{f}_t=\phi f_t-\sum a_{i,j}(\partial_i\phi)(\partial_j
u_t)-\sum(\partial_i\phi)g_{i,t}\, ,$$
$$\widetilde{g_{i,t}}=\phi g_{i,t}-u_t\sum a_{i,j}\partial_j\phi\, ,\quad\widetilde{h_{j,t}}=\phi h_{j,t}\, .$$

Now we prove that $\phi\nu$ is a regular measure:\\ We know that:
\begin{equation}\label{regularmeas}\forall\varphi\in\mathcal{W}_T^+,\quad\int(-\frac{\partial\varphi_s}{\partial s},v_s)ds
+\int\mathcal{E}(\varphi_s,v_s)ds=\int\int\varphi(s,x)d\nu.\end{equation}
We replace $\varphi$ by $\phi\varphi$ in (\ref{regularmeas}), where
$\phi$ is the same as before, and we obtain the following relation:
$$\int(-\frac{\partial\phi\varphi_s}{\partial s},v_s)ds+\int\mathcal{E}(\phi\varphi_s,v_s)ds=\int\int\phi\varphi(s,x)d\nu$$
note that $\phi$ does not depend on $t$ and by a similar calculation
as before,  we get
$$\int(-\frac{\partial\varphi_s}{\partial s},\phi v_s)ds+\int\mathcal{E}(\varphi_s, \phi v_s)ds+\int(K_s,\varphi_s)ds-\int(k_s,\nabla\varphi_s)ds
=\int\int\varphi(s,x)d\phi\nu$$ where $$K_t=\sum
a_{i,j}(\partial_i\phi)(\partial_j v_t),\quad k_t=v_t\sum
a_{i,j}\partial_j\phi.$$ We denote by $\bar{z}$ the solution of the
following PDE with Dirichlet boundary condition and the initial
value $0$:
$$d\bar{z}_t+A\bar{z}_tdt=K_tdt+div k_tdt.$$
If we set $\bar{v}=\phi v+\bar{z}$, then $\bar{v}$ satisfies the
following relation:
$$\int_0^t(-\frac{\partial\varphi_s}{\partial s},\bar{v}_s)ds+\int_0^t\mathcal{E}(\varphi_s,\bar{v}_s)ds=\int_0^t\int_\cO\varphi(x,s)d\phi\nu.$$
It is easy to verify that $\bar{v}\in\mathcal{P}$. Thus $\phi\nu$ is
a regular measure associated to $\bar{v}$.
\\Hence, we deduce that $(\phi u,\phi\nu)$ satisfies an OSPDE with $\phi\xi$ as initial data and zero Dirichlet boundary conditions.
\\Now, we approximate the function $\psi:\ y\in\mathbb{R}\rightarrow\varphi(y^+)$ by a sequence $(\psi_n)$ of regular functions. Let $\zeta$ be a $\mathcal{C}^\infty$ increasing function such that
$$\forall y\in]-\infty,1],\ \zeta(y)=0\ and\ \forall y\in[2,+\infty[,\ \zeta(y)=1.$$
We set for all $n$:$$\forall y\in\mathbb{R},\ \
\psi_n(y)=\varphi(y)\zeta(ny).$$ It is easy to verify that
$(\psi_n)$ converges uniformly to the function $\psi$, $(\psi'_n)$
converges everywhere to the function $(y\rightarrow\varphi'(y^+))$
and $(\psi_n'')$ converges everywhere to the function $(y\rightarrow
\1_{\{y>0\}}\varphi''(y^+))$. Moreover we have the estimates:
\begin{equation}\label{controlofpsi}
\forall y\in\mathbb{R}^+,\ n\in\mathbb{N}^*,\ \
0\leq\psi_n(y)\leq\psi(y),\ \ 0\leq\psi'_n(y)\leq Cy,\ \
\left|\psi_n''(y)\right|\leq C,
\end{equation}
where $C$ is a constant. Thanks to It\^o's formula for the solution
of OSPDE (\ref{SPDEO}) (see Theorem 5 in \cite{DMZ12}), we have
almost surely, for $t\in[0,T]$,
\begin{eqnarray*}
&&\int_\cO\psi_n(w_t(x))dx+\int_0^t\mathcal{E}(\psi'_n(w_s),w_s)ds=\int_\cO\psi_n(\phi(x)\xi(x))dx+\int_0^t\int_\cO\psi'_n(w_s(x))\bar{f}_s(x)dxds\\&&-\sum\int_0^t\int_\cO\psi_n''(w_s(x))\partial_iw_s(x)\widetilde{g_{i,s}}(x)dxds+\sum\int_0^t\int_\cO\psi'_n(w_s(x))\widetilde{h_{j,s}}(x)dxdB_s^j\\&&+\frac{1}{2}\int_0^t\int_\cO\psi_n''(w_s(x))|\widetilde{h_{j,s}}(x)|^2dxds+\int_0^t\int_\cO\psi'_n(w_s(x))d\phi\nu(x,s).
\end{eqnarray*}
Making $n$ tends to $+\infty$ and using the fact that
$\1_{\{w_s>0\}}\partial_iw_s=\partial_iw_s^+$, we get by the
dominated convergence theorem:
 \begin{eqnarray*}
 &&\int_\cO\varphi(w_t^+(x))dx+\int_0^t\mathcal{E}(\varphi'(w^+_s),w^+_s)ds=\int_\cO\varphi(\phi(x)\xi^+(x))dx+\int_0^t\int_\cO\varphi'(w^+_s(x))\bar{f}_s(x)dxds\\&&-\sum\int_0^t\int_\cO\varphi''(w^+_s(x))\partial_iw^+_s(x)\widetilde{g_{i,s}}(x)dxds+\sum\int_0^t\int_\cO\varphi'(w^+_s(x))\widetilde{h_{j,s}}(x)dxdB_s^j\\&&+\frac{1}{2}\int_0^t\int_\cO\varphi''(w^+_s(x))\1_{\{w_s>0\}}|\widetilde{h_{j,s}}(x)|^2dxds+\int_0^t\int_\cO\phi\varphi'(w^+_s(x))d\nu(x,s), \quad a.s.\end{eqnarray*}
Then we consider a sequence $(\phi_n)$ in
$\mathcal{C}_c^\infty(\cO)$, $0\leq\phi_n\leq1$, converging to 1
everywhere on $\cO$ and such that for any $y\in H_0^1(\cO)$  the
sequence $(\phi_ny)$ tends to $y$ in $H_0^1(\cO)$ and
$$\sup_n\left\|\phi_ny\right\|_{H_0^1(\cO)}\leq C\left\|
y\right\|_{H_0^1(\cO)},$$ where $C$ is a constant which does not
depend on $y$. Such a sequence $(\phi_n )$ exists because $\partial
\cO$ is assumed to be Lipschitz  (see Lemma 19 in \cite{DM11}).
\\One has to remark that if $i\in\{1,...d\}$ and $y\in H_0^1(\cO)$,
then $(y\partial_i\phi_n)$ tends to 0 in $L^2(\cO)$.
\\Now, we set $w_n=\phi_n u$ and $$\widetilde{f^n_t}=\phi_n f_t-\sum a_{i,j}(\partial_i\phi_n)(\partial_j u_t)-\sum(\partial_i\phi_n)g_{i,t}$$
$$\widetilde{g^n_{i,t}}=\phi_n g_{i,t}-u_t\sum a_{i,j}\partial_j\phi_n,\quad\widetilde{h^n_{j,t}}=\phi_nh_{j,t}$$
Applying the above It\^o formula to $\varphi(w_n^+)$, we get
 \begin{equation}{\label{It\^o's5}}\begin{split}&\int_\cO\varphi(w_{n,t}^+(x))dx+\int_0^t\mathcal{E}(\varphi'(w^+_{n,s}),w^+_{n,s})ds=\int_\cO\varphi(\phi_n(x)\xi^+(x))dx+\int_0^t\int_\cO\varphi'(w^+_{n,s}(x))\bar{f}_s(x)dxds\\&-\sum\int_0^t\int_\cO\varphi''(w^+_{n,s}(x))\partial_iw^+_{n,s}(x)\widetilde{g_{i,s}}(x)dxds+\sum\int_0^t\int_\cO\varphi'(w^+_{n,s}(x))\widetilde{h_{j,s}}(x)dxdB_s^j\\&+\frac{1}{2}\int_0^t\int_\cO\varphi''(w^+_{n,s}(x))\1_{\{w_{n,s}>0\}}|\widetilde{h_{j,s}}(x)|^2dxds+\int_0^t\int_\cO\phi_n\varphi'(w_{n,s}^+(x))d\nu(x,s),\quad a.s.\end{split}\end{equation}
We have
\begin{eqnarray*}\varphi'(w^+_{n,s})\bar{f}^n_s&-&\sum\varphi''(w^+_{s,n})\partial_iw^+_{n,s}\widetilde{g^n_{i,s}}=\varphi'(w^+_{n,s})\phi_nf_s-\sum
a_{i,j}\varphi'(w^+_{n,s})\partial_j\phi_n\partial_iu_s^+\\&+&\sum
a_{i,j}\varphi''(w^+_{n,s})u^+_s\partial_iw^+_{n,s}\partial_j\phi_n-\sum(\varphi'(w^+_{n,s}))g_{i,s}\partial_i\phi_n+\varphi''(w^+_{n,s})\phi_ng_{i,s}\partial_iw^+_{n,s}
.\end{eqnarray*} Remarking that for all $s\in(0,T]$,
$(\phi_n\varphi'(w^+_{n,s}))$ (resp.
$(\partial_i\phi_n\varphi'(w^+_{n,s}))$) tends to $\varphi'(u^+_s)$
(resp. 0) in $H_0^1(\cO)$ (resp. $L^2(\cO)$) we get by the dominated
convergence theorem the convergence of all the terms  in equality
\eqref{It\^o's5} excepted the one involving the measure $\nu$. For
this last term, we know that $w_n$ is quasi-continuous and from
\eqref{controlofpsi} and \eqref{It\^o's5} it is easy to verify
$$\sup_{n}\int_0^t\int_\cO\phi_n\varphi'(w_{n,s}^+(x))d\nu(x,s)\leq C.$$
Then, by Fatou's lemma, we have
$$\int_0^t\int_\cO\varphi'(u^+_s)\nu(dxds)=\liminf_{n\rightarrow\infty}\int_0^t\int_\cO\phi_n\varphi'(w_{n,s}^+(x))d\nu(x,s)<+\infty,\ \ a.s.$$
Hence, the convergence of the last term comes from the dominated
convergence theorem.
\end{proof}

\subsection{The comparison theorem for local solutions}
Firstly, we prove an It\^o formula for the difference of local
solutions of two OSPDE, $(u^1,\nu^1)\in\cR_{loc}(\xi^1,f^1,g,h,S^1)$
and $(u^2,\nu^2)\in\cR_{loc}(\xi^2,f^2,g,h,S^2)$, where
$(\xi^i,f^i,g,h,S^i)$ satisfy assumptions {\bf(H)}, {\bf(HIL)},
{\bf(OL)} and {\bf(HOL)}. We denote by $\hat{u}=u^1-u^2$,
$\hat{\nu}=\nu^1-\nu^2$, $\hat{\xi }=\xi^1-\xi ^2,$ and
$$ \hat{f}\left( t,\omega ,x,y,z\right) =f^1\left( t,\omega
,x,y+u_t^2\left( x\right) ,z+\nabla u_t^2\left( x\right) \right)
-f^2\left( t,\omega ,x,u_t^2\left( x\right) ,\nabla u_t^2\left(
x\right) \right) , $$ $$ \hat{g}\left( t,\omega ,x,y,z\right)
=g\left( t,\omega ,x,y+u_t^2\left( x\right) ,z+\nabla u_t^2\left(
x\right) \right) -g\left( t,\omega ,x,u_t^2\left( x\right) ,\nabla
u_t^2\left( x\right) \right) , $$ $$ \hat{h}\left( t,\omega
,x,y,z\right) =h\left( t,\omega ,x,y+u_t^2\left( x\right) ,z+\nabla
u_t^2\left( x\right) \right) -h\left( t,\omega ,x,u_t^2\left(
x\right) ,\nabla u_t^2\left( x\right) \right) . $$
\begin{proposition}
Assume that $\partial\cO$ is Lipschitz and that $\hat{u}^+$ belongs
to $\cH_T$.
%assume that the data satisfy the following integrability conditions
%$$ E\left\| \hat{\xi} ^{+}\right\| _2^2<\infty ,\;E\left( \left\|
%\hat{f}^{u,0}\right\| ^*_{\#;t}\right) ^2<\infty ,\;E\left\|
%\hat{g}^{u,0}\right\| _{2,2;t}^2<\infty ,\;E\left\| \hat{h}^{u,0}\right\|
%_{2,2;t}^2<\infty , $$ for each $t\ge 0.$
Let $\varphi :{\mathbb{R}}\rightarrow {\mathbb{R}}$ be a function of
class ${\cal C}^2,$ which admits a bounded second order derivative
and such that $\varphi ^{\prime }\left( 0\right) =0.$ Then the
following relation holds for each $t\in [0,T],$
\begin{eqnarray}\label{It\^o'spositivepartdiff}&&\int_\mathcal{O}\varphi(\hat{u}_t^+(x))dx+\int_0^t\mathcal{E}(\varphi'(\hat{u}_s^+),\hat{u}_s^+)ds=\int_\mathcal{O}\varphi(\hat{\xi}^+(x))dx
+\int_0^t\int_\cO\varphi'(\hat{u}_s^+(x))\hat{f}_s(x)dxds\nonumber\\&&-\sum_{i=1}^d\int_0^t\int_\cO\varphi''(\hat{u}_s^+(x))\partial_i\hat{u}_s^+(x)\hat{g}_s^i(x)dxds+\frac{1}{2}\int_0^t\int_\cO\varphi''(\hat{u}_s^+(x))\1_{\{\hat{u}_s>0\}}|\hat{h}_s(x)|^2dxds\nonumber\\&&+\sum_{i=1}^\infty\int_0^t\int_\cO\varphi'(\hat{u}_s^+(x))\hat{h}_s^j(x)dxdB_s^j+\int_0^t\int_\mathcal{O}\varphi'(\hat{u}_s^+(x))\hat{\nu}(dxds)\qquad
a.s.\end{eqnarray}
\end{proposition}

\begin{proof}We consider  $\phi\in\mathcal{C}_c^\infty(\cO),\ 0\leq\phi\leq1$, and put $$\forall t\in[0,T],\quad \hat{w}_t=\phi \hat{u}_t.$$
From the proof of Proposition \ref{It\^o'slocalposi}, we know that
$(\phi u^1,\phi\nu^1)$ and $(\phi u^2,\phi\nu^2)$ are the solutions
of problem (\ref{SPDEO}) with null Dirichlet boundary conditions. We
have the It\^o formula for $\hat{w}$, see Theorem 6 in \cite{DMZ12}.
Then we do the same approximations as in the proof of Proposition
\ref{It\^o'slocalposi}, we can get the desired formula.
\end{proof}

We have the following comparison theorem:
\begin{theorem}{\label{comparison}}
Assume that $(\xi^i,f^i,g,h,S^i)$, $i=1,2$, satisfy assumptions
{\bf(H)}, {\bf(HIL)}, {\bf(OL)} and {\bf(HOL)}.
 Let $(u^i,\nu^i)\in {\cal R}_{loc}\left( \xi ^i,f^i,g,h,S^i\right)
,i=1,2$ and suppose that the process $\left( u^1-u^2\right) ^{+}$
belongs to $\HH_T
$ and that one has%
$$ E\left( \left\| f^1\left(.,., u^2,\nabla u^2\right) -f^2\left(.,.,
u^2,\nabla u^2\right) \right\| ^*_{\theta;t}\right) ^2<\infty ,\;\;
\mbox{ for all} \quad t\in[0,T].$$ If $\xi ^1\le \xi ^2$ a.s.,
$f^1\left(t,\omega, u^2,\nabla u^2\right) \le f^2\left(t,\omega,
u^2,\nabla u^2\right) $,  $dt\otimes dx \otimes dP$-a.e. and
$S^1\leq S^2$, $dt\otimes dx\otimes dP$-a.s., then one has $u^1
(t,x)\le u^2 (t,x)$, $dt \otimes dx\otimes dP$-a.e.
\end{theorem}
\begin{proof}
%We begin with the It\^o's formula for the difference of two solutions which has been proved in \cite{DMZ12} and do the similar calculus as in the previous %section, we can obtain:
%\begin{eqnarray}\label{It\^o's}&&\int_\mathcal{O}\varphi(\hat{u}_t^+(x))dx+\int_0^t\mathcal{E}(\varphi'(\hat{u}_s^+),\hat{u}_s^+)ds=\int_\mathcal{O}
%\varphi(\hat{\xi}^+(x))dx
%+\int_0^t\int_\cO\varphi'(\hat{u}_s^+(x))\hat{f}_s(x)dxds\nonumber\\&&-\sum_{i=1}^d\int_0^t\int_\cO\varphi"(\hat{u}_s^+(x))\partial_i\hat{u}_s^+(x)
%\hat{g}_s^i(x)dxds+\frac{1}{2}\int_0^t\int_\cO\varphi"(\hat{u}_s^+(x))I_{\{\hat{u}_s>0\}}|\hat{h}_s(x)|^2dxds\nonumber\\&&+
%\sum_{i=1}^\infty\int_0^t\int_\cO\varphi'(\hat{u}_s^+(x))\hat{h}_s^j(x)dxdB_s^j+\int_0^t\int_\mathcal{O}\varphi'(\hat{u}_s^+(x))\hat{\nu}(dxds),\qquad
%a.s.\end{eqnarray}

Applying It\^o's formula (\ref{It\^o'spositivepartdiff}) to
$(\hat{u}^+)^2$, we have $\forall t\in[0,T]$,
\begin{eqnarray}\label{itoposipart}&&\int_\mathcal{O}(\hat{u}_t^+(x))^2dx+2\int_0^t\mathcal{E}((\hat{u}_s^+))ds=\int_\mathcal{O}(\hat{\xi}^+(x))^2dx
+2\int_0^t\int_\cO\hat{u}_s^+(x)\hat{f}_s(x,\hat{u}_s(x),\nabla\hat{u}_s(x))dxds\nonumber\\&&-2\sum_{i=1}^d\int_0^t\int_\cO\partial_i\hat{u}_s^+(x)\hat{g}_s^i(x,\hat{u}_s(x),\nabla\hat{u}_s(x))dxds+\int_0^t\int_\cO
\1_{\{\hat{u}_s>0\}}|\hat{h}_s(x,\hat{u}_s(x),\nabla\hat{u}_s(x))|^2dxds\nonumber\\&&+2\sum_{i=1}^\infty\int_0^t\int_\cO\hat{u}_s^+(x)\hat{h}_s^j(x,\hat{u}_s(x),\nabla\hat{u}_s(x))dxdB_s^j+2\int_0^t\int_\mathcal{O}\hat{u}_s^+(x)\hat{\nu}(dxds),\qquad
a.s.\end{eqnarray} Remarking the following relation
\begin{eqnarray*}\int_0^t\int_\mathcal{O}\hat{u}_s^+(x)\hat{\nu}(dxds)=\int_0^t\int_\mathcal{O}(S^1-u^2)^+\nu^1(dxds)-\int_0^t\int_\mathcal{O}(u^1-S^2)^+\nu^2(dxds)\leq0\end{eqnarray*}
The Lipschitz conditions in $\hat{g}$ and $\hat{h}$ and
Cauchy-Schwarz's inequality lead the following relations: for
$\delta$, $\epsilon>0$, we have
\begin{eqnarray*}
\int_0^t(\nabla
\hat{u}^+_s,\hat{g}_s(\hat{u}_s,\nabla\hat{u}_s)ds\leq(\alpha+\epsilon)\left\|\nabla
\hat{u}^+\right\|^2_{2,2;t}+c_\epsilon\left\|\hat{u}^+\right\|^2_{2,2;t}+c_\epsilon\left\|\hat{g}^{\hat{u},0}\right\|
_{2,2;t}^2,\end{eqnarray*} and
\begin{eqnarray*}
\int_0^t\left\|\1_{\{\hat{u}_s>0\}}\hat{h}_s(\hat{u}_s,\nabla\hat{u}_s)\right\|^2ds\leq(\beta^2+\epsilon)\left\|\nabla\hat{u}^+\right\|^2_{2,2;t}+c_\epsilon\left\|\hat{u}^+\right\|^2_{2,2;t}+c_\epsilon\left\|\hat{h}^{\hat{u},0}\right\|^2_{2,2;t}.
\end{eqnarray*}
Moreover, the Lipschitz condition in $\hat{f}$, the duality relation between elements in  $L_{\theta;t}$ and $L_{\theta;t}^*$ \eqref{dual2} 
and Young's inequality \eqref{young} yield the following relation: 
\begin{eqnarray*}\int_0^t\left(\hat{u}_s^+,\hat{f}_s\left( \hat{u}_s,\hat{u}_s\right) \right)
ds\leq\epsilon \left\| \nabla \hat{u}^+\right\|
_{2,2;t}^2+c_\epsilon \left\|\hat{u}^+\right\| _{2,2;t}^2+\delta
\left\|\hat{u}^+\right\| _{\theta;t}^2+c_\delta \left( \left\|
\hat{f}^{\hat{u},0+}\right\| ^*_{\theta;t}\right) ^2.
\end{eqnarray*}
Since $\cE(\hat{u}^+)\geq\lambda\left\|\nabla\hat{u}^+\right\|^2_2$,
we deduce from \eqref{itoposipart} that for all $t\in[0,T]$, almost
surely,
\begin{eqnarray}\label{8}
\left\|\hat{u}^+_t\right\|_2^2&+&2\left( \lambda -\alpha-\frac{\beta
^2}2-\frac 52\epsilon \right) \left\| \nabla\hat{u}^+\right\|
_{2,2;t}^2\le \left\|\hat{\xi}^+\right\|_2^2+\delta
\left\|\hat{u}^+\right\| _{\theta;t}^2+2c_\delta \left(
\left\|\hat{f}^{\hat{u},0+}\right\| ^*_{\theta;t}\right)
^2\nonumber\\&+&2c_\epsilon \left\|\hat{g}^{\hat{u},0}\right\|
_{2,2;t}^2+c_\epsilon \left\|\hat{h}^{\hat{u},0}\right\|
_{2,2;t}^2+5c_\epsilon \left\|\hat{u}^+\right\| _{2,2;t}^2+2M_t,
\end{eqnarray}
where
$M_t:=\sum_{j=1}^{\infty}\int_0^t\left(\hat{u}^+_s,\hat{h}_s^j\left(
\hat{u}_s,\nabla\hat{u}_s\right) \right) dB_s^j$ represents the
martingale part. Further, using a stopping procedure while taking
the expectation, the martingale part vanishes, so that
\begin{eqnarray*}
&&E\left\|\hat{u}^+_t\right\|_2^2+2\left( \lambda -\alpha
-\frac{\beta ^2}2-\frac 52\epsilon \right) E\left\|
\nabla\hat{u}^+\right\| _{2,2;t}^2\le
E\left\|\hat{\xi}^+\right\|_2^2+\delta E\left\|\hat{u}^+\right\|
_{\theta;t}^2\\&&+2c_\delta E\left(
\left\|\hat{f}^{\hat{u},0+}\right\| ^*_{\theta;t}\right)
^2+2c_\epsilon E\left\|\hat{g}^{\hat{u},0}\right\|
_{2,2;t}^2+c_\epsilon E\left\|\hat{h}^{\hat{u},0}\right\|
_{2,2;t}^2+5c_\epsilon \int_0^tE\left\|\hat{u}^+_s\right\|^2_2ds .
\end{eqnarray*}
Then we choose
$\epsilon=\frac{1}{5}\left(\lambda-\alpha-\frac{\beta^2}{2}\right)$,
set $\gamma=\lambda-\alpha-\frac{\beta^2}{2}$ and apply Gronwall's
lemma obtaining
\begin{equation}\label{1star}
E\left\|\hat{u}^+_t\right\|^2_2+\gamma
E\left\|\nabla\hat{u}^+\right\|^2_{2,2;t}\leq\left(\delta
E\left\|\hat{u}^+\right\|^2_{\theta;t}+ E\left[F\left(\delta,
\hat{\xi}^+,\hat{f}^{\hat{u},0+},\hat{g}^{\hat{u},0},\hat{h}^{\hat{u},0},t\right)\right]\right)e^{5c_\epsilon t},
\end{equation}
with $F(\delta,
\hat{\xi}^+,\hat{f}^{\hat{u},0+},\hat{g}^{\hat{u},0},\hat{h}^{\hat{u},0},t)=\big(\|\hat{\xi}^+\|^2+2c_\delta\big(
\|\hat{f}^{\hat{u},0+}\|
^*_{\theta;t}\big)^2+2c_\epsilon\|\hat{g}^{\hat{u},0}\|_{2,2;t}^2+c_\epsilon\|\hat{h}^{\hat{u},0}\|_{2,2;t}^2\big)$.
As a consequence one gets
\begin{equation}\label{2star}
E\left\|\hat{u}^+\right\|^2_{2,2;t}\leq\frac{1}{5c_\epsilon}\left(\delta
E\left\|\hat{u}^+\right\|^2_{\theta;t}+E\left[F\left(\delta,
\hat{\xi}^+,\hat{f}^{\hat{u},0+},\hat{g}^{\hat{u},0},\hat{h}^{\hat{u},0},t\right)\right]\right)\left(e^{5c_\epsilon
t}-1\right) .
\end{equation}
Now we return to the inequality (\ref{8}) and take the supremum in time,
getting
\begin{equation}\label{supremum}
\left\|\hat{u}^+\right\|^2_{2,\infty;t}\leq\delta\left\|\hat{u}^+\right\|^2_{\theta;t}+F\left(\delta,
\hat{\xi}^+,\hat{f}^{\hat{u},0+},\hat{g}^{\hat{u},0},\hat{h}^{\hat{u},0},t\right)+5c_\epsilon\left\|\hat{u}^+\right\|^2_{2,2;t}+2\sup_{s\leq
t}M_s
\end{equation}
We would like to take the expectation in this relation and for that
reason we need to estimate the bracket of the martingale part,
\begin{eqnarray*} \left\langle M\right\rangle _t^{\frac 12}\le \left\|\hat{u}^+\right\|
_{2,\infty ;t}\left\|\hat{h}(\hat{u},\nabla\hat{u})\right\|
_{2,2;t}\le \eta \left\|\hat{u}^+\right\| _{2,\infty ;t}^2+c_\eta
\left( \left\| \hat{u}^+\right\| _{2,2;t}^2+\left\|
\nabla\hat{u}^+\right\|
_{2,2;t}^2+\left\|\hat{h}^{\hat{u},0}\right\|^2_{2,2;t}\right)
\end{eqnarray*} with $\eta $ another small parameter to be properly
chosen. Using this estimate and the inequality of
Burkholder-Davis-Gundy we deduce from the inequality
(\ref{supremum}):
$$ \left( 1-2C_{BDG}\eta \right) E\left\|\hat{u}^+\right\| _{2,\infty
;t}^2\le \delta E\left\|\hat{u}^+\right\|
_{\theta;t}^2+E\left[F\left(\delta,
\hat{\xi}^+,\hat{f}^{\hat{u},0+},\hat{g}^{\hat{u},0},\hat{h}^{\hat{u},0},t\right)\right]$$
$$ +\left( 5c_\varepsilon +2C_{BDG}c_\eta
\right) E\left\|\hat{u}^+\right\| _{2,2;t}^2+2C_{BDG}c_\eta E\left\|
\nabla \hat{u}^+\right\| _{2,2;t}^2+2C_{BDG}c_\eta
E\left\|\hat{h}^{\hat{u},0}\right\|^2_{2,2;t}$$ where $C_{BDG}$ is
the constant corresponding to the Burkholder-Davis-Gundy inequality.
Further we choose the parameter $\eta =\frac 1{4C_{BDG}}$
and combine this estimate with \eqref{1star} and \eqref{2star} to deduce an estimate of the form:
\begin{eqnarray*}
E\left( \left\|\hat{u}^+\right\| _{2,\infty ;t}^2+\left\| \nabla
\hat{u}^+\right\| _{2,2;t}^2\right) \le \delta c_2\left( t\right)
E\left\| \hat{u}^+\right\|
_{\theta;t}^2+c_3(\delta,t)E\left[R\left(\delta,
\hat{\xi}^+,\hat{f}^{\hat{u},0+},\hat{g}^{\hat{u},0},\hat{h}^{\hat{u},0},t\right)\right]
\end{eqnarray*}
where $R\left(\delta,
\hat{\xi}^+,\hat{f}^{\hat{u},0+},\hat{g}^{\hat{u},0},\hat{h}^{\hat{u},0},t\right)=\left(\left\|\hat{\xi}^+\right\|^2+\left(
\left\|\hat{f}^{\hat{u},0+}\right\|
^*_{\theta;t}\right)^2+\left\|\hat{g}^{\hat{u},0}\right\|_{2,2;t}^2+\left\|\hat{h}^{\hat{u},0}\right\|_{2,2;t}^2\right)$
and $c_3(\delta,t)$ is a constant that depends on $\delta $ and $t,$
while $c_2\left( t\right) $ is independent of $ \delta .$ Dominating
the term $E\left\|\hat{u}^+\right\| _{\theta;t}^2$ by using the
estimate (\ref{controltheta}) and then choosing $\delta =\frac
1{2c_1^2c_2\left( t\right) }$, we get the following estimate:
%For the martingale part, we use the Burkholder-Davies-Gundy's inequality and get:
%\begin{eqnarray*}
%&&E\sup_{s\leq t}\left|\sum_{j=1}^{+\infty}\int_0^s((u_r-S'_r)^+,\bar{h}^j_r (u_r-S'_r ,\nabla (u_r-S'_r) ))dB^j_r\right|\\&&\leq c_\eta E[\int_0^t
%\sum_{j=1}%^{+\infty}((u_r-S'_r)^+,\bar{h}^j_r((u_r-S'_r)^+,\nabla(u_r-S'_r)^+)^2dr]^{1/2}\\&&
%\leq c_\eta E[\int_0^t\sum_{j=1}^{+\infty}\sup_{r\leq t}\left\|(u_r-S'_r)^+\right\|^2\left\|\bar{h}^j_r ((u_r-S'_r)^+ ,\nabla(u_r-S'_r)^+)\right\|^2dr]^{1/2}
%\\&&\leq c_\eta E[\sup_{s\leq t}\left\|(u_s-S'_s)^+\right\|^2(\int_0^t\left\|\bar{h}%_r ((u_r-S'_r)^+ ,\nabla (u_r-S'_r)^+)\right\|^2dt)^{1/2}]
%\\&&\leq\epsilon E\sup_{s\leq t}\left\|(u_s-S'_s)^+\right\|^2+c_\epsilon E\int_0^t\left\|\bar{h}_r ((u_r-S'_r)^+,\nabla(u_r-S'_r)^+)\right\|^2ds.
%\end{eqnarray*}
\begin{eqnarray*}
E\left(\left\|\hat{u}^+\right\|^2_{2,\infty;t}+\left\|\nabla\hat{u}^+\right\|^2_{2,2;t}\right)\leq
k(t)E\left(\left\|\hat{\xi}^+\right\|^2_2+\left(\left\|\hat{f}^{\hat{u},0+}\right\|^*_{\theta;t}\right)^2+\left\|\hat{g}^{\hat{u},0}\right\|^2_{2,2;t}+\left\|\hat{h}^{\hat{u},0}\right\|^2_{2,2;t}\right).
\end{eqnarray*}

%$$ E\left( \left\| \hat{u}^{+}\right\| _{2,\infty ;t}^2+\left\| \nabla
%\hat{u}^{+}\right\| _{2,2;t}^2\right) \le k\left( t\right) E\left(
%\left\| \hat{\xi}^{+}\right\| _2^2+\left( \left\|
%\hat{f}^{\hat{u},0+}\right\| ^*_{\theta;t}\right) ^2+\left\|
%\hat{g}^{\hat{u},0}\right\| _{2,2;t}^2+\left\| \hat{h}^{\hat{u},0}\right\|
%_{2,2;t}^2\right) . $$

This implies the desired result since $\hat{\xi}\leq0$, $\hat{f}^0\leq0$
and $\hat{g}^0=\hat{h}^0=0$.
\end{proof}
\subsection{Maximum principle}\label{maxprinciple}
We first consider the case of a solution $u$ such that $u\leq0$ on
$\partial\cO$.
\begin{theorem}
\label{maxprinc} Suppose that Assumptions {\bf (H)}, {\bf (OL)},
{\bf (HIL)}, {\bf (HOL)},{\bf (HI$\mathbf{2 p}$)}, {\bf
(HO$\mathbf{\infty p}$)} and {\bf (HD$\mathbf{\theta p}$)} hold  for
some $\theta \in [0,1[$, $p\geq2$ and that the constants of the
Lipschitz conditions satisfy
$$\alpha +\frac{\beta ^2}2+72\beta ^2<\lambda. $$ Let $(u,\nu)\in
{\cal R}_{loc}\left( \xi ,f,g,h,S\right) $ be such that
$u^{+}\in \HH.$  Then one has%
\begin{eqnarray*} E\left\| u^{+}\right\| _{\infty ,\infty ;t}^p&\le &k(t)c(p)E\big(\left\|\xi^+-S'_0\right\|^p_\infty+(\left\| \bar{f}^{0,+}\right\|^*_{\theta;t})^p+
(\left\||\bar{g}^0|^2\right\|^*_{\theta;t})^{\frac{p}{2}}+(\left\||\bar{h}^0|^2\right\|^*_{\theta;t})^{\frac{p}{2}}\\&+&\left\|(S'_0)^+\right\|^p_\infty+(\left\|f^{',+}\right\|^*_{\theta;t})^p+
(\left\||g'|^2\right\|^*_{\theta;t})^{\frac{p}{2}}+(\left\||h'|^2\right\|^*_{\theta;t})^{\frac{p}{2}}\big)\end{eqnarray*}
where $k\left( t\right) $ is constant that depends on the structure
constants and $t\in [0,T].$
\end{theorem}

%For $p=2k$ and $t\geq0$ noting the following relation
%$$\int\int\left((u-S')^+\right)^{p-1}\nu(dxds)=\int\int\left((S-S')^+\right)^{p-1}\nu(dxds)=0$$ we
%have
%\begin{eqnarray*}E\left\|(u-S')^+\right\|_{\infty,\infty;t}^p\leq k(t)E(\left\|\xi^+-S'_0\right\|^p_\infty+(\left\| \bar{f}^{0,+}\right\|^*_{\theta;t})^p+
%(\left\||\bar{g}^0|^2\right\|^*_{\theta;t})^{p/2}+(\left\||\bar{h}^0|^2\right\|^*_{\theta;t})^{p/2}).\end{eqnarray*}where $\bar{f}$, $\bar{g}$ and $\bar{h}$ are the same as in Section \ref{LPestimate}.
\begin{proof}Set $(y,\nu')=\mathcal{R}(\xi^+,\check{f},g,h,S)$ the solution with
zero Dirichlet boundary conditions, where the function $\check{f}$
is defined by $\check{f}=f+f^{0,-}$, with $f^{0,-}=0\vee(-f^0)$. The
assumption on the Lipschitz constants ensure the application of the
Section \ref{LPestimate}, which gives the following estimate :
\begin{eqnarray*}E\left\| y-S'\right\|^p_{\infty,\infty;t}\leq k(t)E(\left\|\xi^+-S'_0\right\|^p_\infty+(\left\| \bar{f}^{0,+}\right\|^*_{\theta;t})^p+
(\left\||\bar{g}^0|^2\right\|^*_{\theta;t})^{\frac{p}{2}}+(\left\||\bar{h}^0|^2\right\|^*_{\theta;t})^{\frac{p}{2}})\,,\end{eqnarray*}
where $\bar{f}^{0,+}=\check{f}^0-f'=f^{0,+}-f'$. On the boundary,
$y=0$ and $u\leq0$, hence, $u-y\leq0$ on the boundary, i.e.
$(u-y)^+\in\mathcal{H}$. Moreover, the other conditions of Theorem
\ref{comparison} are satisfied so that we can apply it and deduce
that $u-S'\leq y-S'$. This implies that $(u-S')^+\leq (y-S')^+$ and
the above estimate of $y-S'$ leads to the following estimate:
\begin{eqnarray*}E\left\| (u-S')^+\right\|^p_{\infty,\infty;t}\leq k(t)E(\left\|\xi^+-S'_0\right\|^p_\infty+(\left\| \bar{f}^{0,+}\right\|^*_{\theta;t})^p+
(\left\||\bar{g}^0|^2\right\|^*_{\theta;t})^{\frac{p}{2}}+(\left\||\bar{h}^0|^2\right\|^*_{\theta;t})^{\frac{p}{2}}).\end{eqnarray*}
with the estimate of $S'$
\begin{eqnarray*}E\left\|(S')^+\right\|^p_{\infty,\infty;t}\leq k(t)E(\left\|(S'_0)^+\right\|^p_\infty+(\left\|f^{',+}\right\|^*_{\theta;t})^p+
(\left\||g'|^2\right\|^*_{\theta;t})^{\frac{p}{2}}+(\left\||h'|^2\right\|^*_{\theta;t})^{\frac{p}{2}}).\end{eqnarray*}
Therefore, \begin{eqnarray*} E\left\| u^{+}\right\| _{\infty ,\infty
;t}^p&\le &k(t)c(p)E\big(\left\|\xi^+-S'_0\right\|^p_\infty+(\left\|
\bar{f}^{0,+}\right\|^*_{\theta;t})^p+
(\left\||\bar{g}^0|^2\right\|^*_{\theta;t})^{\frac{p}{2}}+(\left\||\bar{h}^0|^2\right\|^*_{\theta;t})^{\frac{p}{2}}\\&+&\left\|(S'_0)^+\right\|^p_\infty+(\left\|f^{',+}\right\|^*_{\theta;t})^p+
(\left\||g'|^2\right\|^*_{\theta;t})^{\frac{p}{2}}+(\left\||h'|^2\right\|^*_{\theta;t})^{\frac{p}{2}}\big).\end{eqnarray*}
\end{proof}

Let us generalize the previous result by onsidering a real It\^o
process of the
form$$M_t=m+\int_0^tb_sds+\sum_{j=1}^{+\infty}\int_0^t\sigma_{j,s}dB_s^j$$
where $m$ is a random variable and $b=(b_t)_{t\geq0}$,
$\sigma=(\sigma_{1,t},...,\sigma_{n,t},...)_{t\geq0}$ are adapted
processes.
\begin{theorem}
\label{maintheo} Suppose that Assumptions {\bf (H)}, {\bf (OL)},
{\bf (HIL)}, {\bf (HOL)},{\bf (HI$\mathbf{2 p}$)}, {\bf
(HO$\mathbf{\infty p}$)} and {\bf (HD$\mathbf{\theta p}$)} hold for
some $\theta \in [0,1[$, $p\geq2$ and that the constants of the
Lipschitz conditions satisfy
$$\alpha +\frac{\beta ^2}2+72\beta ^2<\lambda. $$ Assume also that
$m$ and the processes $b$ and $\sigma $
satisfy the following integrability conditions%
$$ E\left| m\right| ^p<\infty ,\;E\left( \int_0^t\left| b_s\right|
^{\frac 1{1-\theta }}ds\right) ^{p\left( 1-\theta \right) }<\infty
,\;E\left(
\int_0^t\left| \sigma _s\right| ^{\frac 2{1-\theta }}ds\right) ^{\frac{%
p\left( 1-\theta \right) }2}<\infty , $$ for each $t\in[0,T].$ Let
$(u,\nu)\in {\cal R}_{loc}\left( \xi ,f,g,h,S\right) $ be such that
$\left( u-M\right) ^{+}$ belongs to $\HH_T$. Then one has
\begin{eqnarray}\label{MaxiPrinc}E\left\|(u-M)^+\right\|^p_{\infty,\infty;t}&\leq& c(p)k(t)E\big[\left\|(\xi-m)^+-(S'_0-m)\right\|_\infty^p+
\left(\left\|\bar{f}^{0,+}\right\|_{\theta;t}^*\right)^p\nonumber\\&+&\left(\left\|\left|\bar{g}^0\right|^2\right\|_{\theta;t}^*\right)^{\frac{p}{2}}
+\left(\left\|\left|\bar{h}^0\right|^2\right\|_{\theta;t}^*\right)^{\frac{p}{2}}+\left\|(S'_0-m)^+\right\|^p_\infty\\&+&\left(\left\|(f'-b)^+\right\|^*_{\theta;t}\right)^p+\left(\left\||g'|^2\right\|^*_{\theta;t}\right)^{\frac{p}{2}}+\left(\left\||h'-\sigma|^2\right\|^*_{\theta;t}\right)^{\frac{p}{2}}\big]\nonumber\end{eqnarray}
where $k\left( t\right) $ is the constant from the preceding
corollary. The right hand side of this estimate is dominated by the
following quantity which is expressed directly in terms of the
characteristics of the process $M$,
\begin{eqnarray*}
&&c(p)k(t)E\big[\left\|(\xi-m)^+-(S'_0-m)\right\|^p_\infty+\left(\left\|\bar{f}^{0,+}\right\|_{\theta;t}^*\right)^p
+\left(\left\|\left|\bar{g}^0\right|^2\right\|_{\theta;t}^*\right)^{\frac{p}{2}}+\left(\left\|\left|\bar{h}^0\right|^2\right\|_{\theta;t}^*\right)^{\frac{p}{2}}\\&&
\quad\qquad\quad+\left\|(S'_0-m)^+\right\|^p_\infty+\left(\left\|f^{',+}\right\|_{\theta;t}^*\right)^p
+\left(\left\|\left|g'\right|^2\right\|_{\theta;t}^*\right)^{\frac{p}{2}}+\left(\left\|\left|h'\right|^2\right\|_{\theta;t}^*\right)^{\frac{p}{2}}
\\&&\quad\qquad\quad+\left(\int_0^t\left|b_s\right|^{\frac{1}{1-\theta}}ds\right)^{p(1-\theta)}+\left(\int_0^t\left|\sigma_s\right|^{\frac{2}{1-\theta}}ds\right)^{\frac{p(1-\theta)}{2}}
\big].
\end{eqnarray*}
\end{theorem}
\begin{proof}
One immediately observes that $u-M$ belongs to $\mathcal{ R}_{loc}\left( \xi -m,%
\check{f},\check{g},\check{h}, S-M\right) ,$ where%
$$ \check{f}\left( t,\omega ,x,y,z\right) =f\left( t,\omega
,x,y+M_t\left( \omega \right) ,z+\nabla M_t(\omega)\right)
-b_t\left( \omega \right) ,
$$ $$ \check{g}\left( t,\omega ,x,y,z\right) =g\left( t,\omega
,x,y+M_t\left( \omega \right) ,z+\nabla M_t(\omega)\right) , $$ $$
\check{h}\left( t,\omega ,x,y,z\right) =h\left( t,\omega
,x,y+M_t\left( \omega \right) ,z+\nabla M_t(\omega)\right) -\sigma
_t\left( \omega \right) . $$ In order to apply the preceding theorem
we only have to estimate the zero terms of the following functions:
$$ \bar{\check{f}}\left( t,\omega ,x,y,z\right) =\check{f}\left( t,\omega
,x,y+S'-M, z+\nabla(S'-M)\right) - f'(t,\omega,x)+b_t(\omega), $$
$$\bar{\check{g}}\left( t,\omega ,x,y,z\right) =\check{g}\left( t,\omega
,x,y+S'-M, z+\nabla(S'-M)\right) - g'(t,\omega,x), $$
$$\bar{\check{h}}\left( t,\omega ,x,y,z\right) =\check{h}\left( t,\omega
,x,y+S'-M, z+\nabla(S'-M)\right) - h'(t,\omega,x)+\sigma_t(\omega).
$$ So we have: $$\bar{\check{f}}_t^0=\check{f}_t(S'-M,\nabla
(S'-M))-f'_t+b_t=f_t(S',\nabla S')-f'_t=\bar{f}^0\, ,$$
$$\bar{\check{g}}_t^0=\check{g}_t(S'-M,\nabla (S'-M))-g'_t=g_t(S',\nabla S')-g'_t=\bar{g}^0\, ,$$
$$\bar{\check{h}}_t^0=\check{h}_t(S'-M,\nabla (S'-M))-h'_t+\sigma_t=h_t(S',\nabla S')-h'_t=\bar{h}^0\, .$$
%Further we may write:
%$$\bar{\check{f}}^{0,+}_t\leq c\left[\left|M_t\right|+\bar{f}^{0,+}_t+f^{',+}_t+\left|f'_t\right|\right]\, ,$$
%$$\left|\bar{\check{g}}^0_t\right|^2\leq c\left[\left|M_t\right|^2+\left|\bar{g}^0_t\right|^2+\left|g'_t\right|^2\right]\, ,$$
%$$\left|\bar{\check{h}}^0_t\right|^2\leq c\left[\left|M_t\right|^2+\left|\bar{h}^0_t\right|^2+|h'_t|^2\right]\, . $$
%Then we have the estimates%
%$$ \left\| \bar{\check{f}}^{0,+}\right\| _{\theta ;t}^{*}\leq c\left[\left\|\bar{f}^{0,+}\right\| _{\theta ;t}^{*}+\left\|f^{',+}\right\|_{\theta;t}^*
%+\left\|f'\right\|_{\theta;t}^*+\sup _{s\le t}\left| M_t\right|\right], $$
%$$\left\| \left|\bar{\check{g}}^0\right| ^2\right\| _{\theta ;t}^{*}\leq c\left[\left\|\left| \bar{g}^0\right| ^2\right\| _{\theta ;t}^{*}+\left\|
%\left|g'\right| ^2\right\| _{\theta ;t}^{*}+\sup _{s\le t}\left|M_t\right| ^2\right], $$
%$$ \left\| \left|\bar{\check{h}}^0\right|^2\right\| _{\theta ;t}^{*}\leq c\left[\left\| \left| \bar{h}^0\right| ^2\right\| _{\theta ;t}^{*}+
%\left\| \left|h'\right| ^2\right\| _{\theta ;t}^{*}+\sup _{s\le t}\left|M_t\right| ^2\right]. $$
Therefore, applying the preceding theorem to $u-M$, we obtain
(\ref{MaxiPrinc}).
\\On the other hand,
%$$ \sup _{s\le t}\left| M_t\right| \le \left| m\right|
%+\int_0^t\left| b_s\right| ds+\sup _{s\le t}\left| N_t\right| , $$
%where we have denoted by $N_t$ the martingale $\sum_{j=1}^{+\infty}\int_0^t%
%\sigma _{j,s}dB_s^j.$ The inequality of Burkholder -Davis -Gundy implies%
%$$ E\sup _{s\le t}\left| M_t\right| ^p\le cE\left[ \left| m\right|
%^p+\left( \int_0^t\left| b_s\right| ds\right) ^p+\left(
%\int_0^t\left| \sigma _s\right| ^2ds\right) ^{\frac p2}\right].$$
one has the following estimates:
$$\left\|(f'-b)^+\right\|_{\theta;t}^*\leq c\left[\left\|f^{',+}\right\|_{\theta;t}^*+\left(\int_0^t\left|b_s\right|^{\frac{1}{1-\theta}}ds\right)^{1-\theta}\right],$$
$$\left\||h'-\sigma|^2\right\|^*_{\theta;t})^{\frac{p}{2}}\leq c\left[\left(\left\||h'|^2\right\|^*_{\theta;t}\right)^{\frac{p}{2}}+\left(\int_0^t\left|\sigma_s\right|^{\frac{2}{1-\theta}}ds\right)^{1-\theta}\right].$$
This allows us to conclude the proof.
\end{proof}

\section{Appendix}
\label{tecnicLem} In this section, we  prove some technical lemmas
that  we need  in the Step 2 of the proof of Theorem \ref{LPESTIM}.
For simplicity, we put, for fixed $n\leq m$,
$\hat{u}:=\bar{u}^n-\bar{u}^m$, $\hat{\xi}:=\xi^n-\xi^m$,
$\hat{f}(t,\omega,x,y,z):=\bar{f}_{n,m}(t,\omega,x,y,z)$ and similar
for $\hat{g}$ and $\hat{h}$.

We recall that the initial value $\hat{\xi}$ and $\hat{f}^0,\
\hat{g}^0,\ \hat{h}^0$ are uniformly bounded.
\begin{lemma}Denote
\begin{eqnarray*}
K=\left\|
\hat{\xi}\right\|_{L^\infty(\Omega\times\cO)}\vee\left\|\hat{f}^0\right\|_{L^\infty(\mathbb{R}_+\times\Omega\times\cO)}\vee\left\|\hat{g}^0\right\|_{L^\infty(\mathbb{R}_+\times\Omega\times\cO)}\vee\left\|\hat{h}^0\right\|_{L^\infty(\mathbb{R}_+\times\Omega\times\cO)
}.\end{eqnarray*} Then there exist constants $c,\ c'>0$ which only
depend on $K,\ C,\ \alpha,\ \beta$ such that, for all real $l\geq2$,
one has
\begin{equation}\label{diffl}E\int_\cO|\hat{u}_t(x)|^ldx\leq cK^2l(l-1)e^{cl(l-1)t},
\end{equation}
\begin{equation}\label{diffl2}E\int_0^t\int_\cO|\hat{u}_s(x)|^{l-2}|\nabla\hat{u}_s(x)|^2dxds\leq c'K^2l(l-1)e^{cl(l-1)t},\end{equation}
and\begin{equation}\label{difflm}E\int_0^t\int_\cO
|\hat{u}_s(x)|^{l-1}(\nu^n+\nu^m)(dxds)<+\infty .\end{equation}
\end{lemma}
\begin{proof}
Beginning from the It\^o formula for the difference of solutions of
two obstacle problems which has been proved in \cite{DMZ12}: we take
the same $\varphi_n$ as in the proof of Lemma \ref{estimateul},
\begin{equation}\label{difflp}
\begin{split}
& \int_{\mathcal{O}} \varphi_n (\hat{u}_t (x)) \, dx \; + \;
\int_0^t \mathcal{E} \big(\varphi_n^{\prime} (\hat{u}_s), \,
\hat{u}_s \big) \, ds=\int_\cO\varphi_n(\hat{\xi}(x))dx +
\int_{0}^{t}\int_{\mathcal{O}}
\varphi_n^{\prime}(\hat{u}_s(x))\, \hat{f}(s,x)\, dx ds \\
&  - \sum_{i=1}^{d} \int_{0}^{t} \int_{\mathcal{O}}
\varphi_n^{\prime \prime } (\hat{u}_s(x))
\partial_i (\hat{u}_s(x)))\,\hat{g}_i(s,x)\, dx \, ds +
\sum_{j=1}^{\infty} \int_{0}^{t} \int_{\mathcal{O}}
\varphi_n^{\prime} (\hat{u}_s(x)) \, \hat{h}_j(s,x)\, dx dB_{s}^{j}\\
& + \frac{1}{2}\, \sum_{j=1}^{\infty} \int_{0}^{t}\int_{\mathcal{O}}
\varphi_n^{\prime \prime } (\hat{u}_s(x))\, \hat{h}_j^2(s,x) \, dx\,
ds
+\int_0^t\int_\cO\varphi_n^{\prime}(\hat{u}_s(x))\, (\nu^n-\nu^m)(dx\, ds)\, ,\quad a.s.\\
\end{split}
\end{equation}
The support of $\nu^n$ is $\{\bar{u}^n=S\}$ and the support of
$\nu^m$ is $\{\bar{u}^m=S\}$, so the last term is equal to
$$\int_0^t\int_\cO\varphi_n^{\prime}(S_s(x)-\bar{u}^m_s(x))\, \nu^n(dx\,
ds)-\int_0^t\int_\cO\varphi_n^{\prime}(\bar{u}^n_s(x)-S_s(x))\,
\nu^m(dx\, ds)$$ and the fact that $\varphi'_n(x)\leq0$, when
$x\leq0$ and $\varphi'_n(x)\geq0$, when $x\geq0$, ensure that the
last term is always negative.
%\begin{eqnarray*}&&\int_0^t\int_\cO\varphi_n^{\prime}(S_s(x)-u^2_s(x))\1_{\{\left|S-u^2\right|\leq
%n\}}\, \nu^1(dx\, ds)\\&=&l\int_0^t\int_\cO
%sgn(S_s-u^2_s)\left|S_s(x)-u^2_s(x)\right|^{l-1}\, \nu^1(dx\,
%ds)\le0, \end{eqnarray*}
%\begin{eqnarray*}
%&&\int_0^t\int_\cO\varphi_n^{\prime}(S_s(x)-u^2_s(x))\1_{\{\left|S-u^2\right|>n\}}\, \nu(dx\, ds)\\
%&=&\int_0^t\int_\cO n^{l-2}sgn(S_s -u^2_s)[l(l-1)(\left|S_s(x)-u^2_s
%(x) \right|-n)+ln]\,\nu^1(dxds)<0,
%\end{eqnarray*}
%\begin{eqnarray*}&&\int_0^t\int_\cO\varphi_n^{\prime}(u^1_s(x)-S_s(x))
%\1_{\{\left|u-S\right|\leq n\}}\, \nu^1 (dx\,
%ds)\\&=&l\int_0^t\int_\cO
%sgn(u^1_s-S_s)\left|u_s(x)-S_s(x)\right|^{l-1}\, \nu^1 (dx\,
%ds)\geq0,\end{eqnarray*} and
%\begin{eqnarray*}
%&&\int_0^t\int_\cO\varphi_n^{\prime}(u^1_s(x)-S_s(x))\1_{\{\left|u-S\right|>n\}}\, \nu^1(dx\, ds)\\
%&=&\int_0^t\int_\cO
%n^{l-2}sgn(u^1_s-S_s)[l(l-1)(\left|u^1_s(x)-S_s(x)\right|-n)+ln]\,\nu^1(dxds)>0.
%\end{eqnarray*}
\\By the uniform ellipticity of the operator $A$, we get
\begin{equation*}
\mathcal{E} \big(\varphi_n^{\prime} (\hat{u}_s), \, \hat{u}_s \big)
\, \geq \, \lambda  \, \int_{\cO} \varphi_n^{\prime \prime}
(\hat{u}_s) |\nabla \hat{u}_s|^2 \, dx.\end{equation*}

Let $\epsilon >0$ be fixed. Using the Lipschitz condition on
$\hat{f}$ and the properties of the functions $(\varphi_n)_n$ we get
\begin{equation*} \begin{split}   & |\varphi_n^{\prime} (\hat{u}_s)| \, |\hat{f}(s,x)|
\leq l (\varphi_n (\hat{u}_s) + 1) \, |\hat{f}^0| + (C +
 c_{\epsilon})\, |\hat{u}_s|^2
 \varphi_n^{\prime\prime}(\hat{u}_s)  + \, \epsilon \varphi_n^{\prime\prime}(\hat{u}_s)|\nabla(\hat{u}_s)|^2 .\\
 \end{split}
 \end{equation*}
 Now using Cauchy-Schwarz inequality and the Lipschitz condition on $\hat{g}$ we get
 \begin{equation*}
 \begin{split}
 & \,\sum_{i=1}^d \varphi_n^{\prime\prime}(\hat{u}_s) \partial_i (\hat{u}_s)
 \, \hat{g}(s,x)  \leq   l(l-1)c_{\epsilon}K^2 + 2c_{\epsilon}(K^2 + C^2 ) l (l-1)
 |\varphi_n (\hat{u}_s)|+
 (\alpha + \epsilon) \, \varphi_n^{\prime\prime} (\hat{u}_s)
 |\nabla (\hat{u}_s)|^2 .  \\
 \end{split}
 \end{equation*}
  In the same way as before
 %and  using the obvious inequality :
 %$ (a + b)^2 \leq c'_{\epsilon} a^2 + (1 + \epsilon) b^2$ for each $ \epsilon > 0 $
\begin{equation*}
 \begin{split}&
\sum_{j=1}^{\infty}  \varphi_n^{\prime \prime } (\hat{u}_s )\,
\hat{h}(s,x) \leq \, 2c'_{\epsilon} l (l-1)K^2 +  2c'_{\epsilon}(
K^2 +C^2 )l (l-1) \varphi_n(\hat{u}_s)
 + (1 + \epsilon) \, \beta^2\, \varphi_n^{\prime \prime }(\hat{u}_s)|\nabla (\hat{u}_s)|^2 .
\end{split}
 \end{equation*}
Thus taking the expectation, we deduce
\begin{equation*}
\begin{split}
&E \, \int_{\mathcal{O}}\varphi_n (\hat{u}_t(x)) \, dx + (\lambda -
\frac{1}{2}(1+ \epsilon)\beta^2 - (\alpha + 2 \epsilon) \, ) \, E \,
\int_0^t \int_{\mathcal{O}} \varphi_n^{\prime \prime} (\hat{u}_s) \,
|\nabla( \hat{u}_s) |^2  \, dx \, ds \\
& \, \leq \, l(l-1)c''_{\epsilon} K^2 \, + \,  c''_{\epsilon} l
(l-1)\big( K^2 +C^2+C+ c_{\epsilon}\big)
E \, \int_0^t \int_{\mathcal{O}} \varphi_n (\hat{u}_s(x))\, dx \, ds . \\
\end{split}
\end{equation*}
On account  of the contraction condition, one can choose $ \epsilon
> 0 $ small enough such that $$ \lambda - \frac{1}{2}(1+
\epsilon)\beta^2 - (\alpha + 2 \epsilon)>0 $$  and then
\begin{equation*}
\begin{split}
E\, \int_{\mathcal{O}} \varphi_n (\hat{u}_t(x)) \, dx    &
 \leq \, c K^2 l (l-1) \, + \,  cl (l-1)
E \, \int_0^t \int_{\mathcal{O}} \varphi_n (\hat{u}_s(x))\, dx \, ds \, .\\
\end{split}
\end{equation*}
 We obtain by Gronwall's Lemma, that
\begin{equation*}
\begin{split}
E\, \int_{\mathcal{O}} \varphi_n (\hat{u}_t(x)) \, dx    &
 \leq \, c\,K^2  l (l-1) \, \exp{ \big( c \, l (l-1)\,t \big) }, \\
\end{split}
\end{equation*}
 and  so it is easy to get
\begin{equation*}
\begin{split}
  E \, \int_0^t \int_{\mathcal{O}} \varphi_n^{\prime \prime} (\hat{u}_s(x))  \,
|\nabla \hat{u}_s |^2  \, dx \, ds  \, \leq \,
 c^{\prime}\, K^2 l \, (l-1) \,  \exp{ \big( c l (l-1)\,t \big)}.\\
\end{split}
\end{equation*}
Then, letting $ n \to \infty $,  by Fatou's lemma we get
(\ref{diffl}) and (\ref{diffl2}).\\From (\ref{difflp}), we know that
\begin{equation*}-\int_0^t\int_\cO\varphi'_n(\hat{u}_s(x))(\nu^n-\nu^m)(dxds)\leq C.\end{equation*}
Moreover,
\begin{eqnarray*}&&-\int_0^t\int_\cO\varphi'_n(\hat{u}_s(x))(\nu^n-\nu^m)(dxds)\\&=&-\int_0^t\int_\cO\varphi_n^{\prime}(S_s(x)-\bar{u}^m_s(x))\,
\nu^n(dx\,
ds)+\int_0^t\int_\cO\varphi_n^{\prime}(\bar{u}^n_s(x)-S_s(x))\,
\nu^m(dx\,
ds)\\&=&\int_0^t\int_\cO\varphi_n^{\prime}(\bar{u}^m_s(x)-S_s(x))\,
\nu^n(dx\,
ds)+\int_0^t\int_\cO\varphi_n^{\prime}(\bar{u}^n_s(x)-S_s(x))\,
\nu^m(dx\, ds)\end{eqnarray*} By Fatou's lemma, we obtain
\begin{eqnarray*}
\int_0^t\int_\cO|\bar{u}^m_s(x)-S_s(x)|^{l-1}\nu^n(dxds)+\int_0^t\int_\cO|\bar{u}^n(x)-S_s(x)|^{l-1}\nu^m(dxds)<+\infty,
\ a.s.
\end{eqnarray*}
which is exactly (\ref{difflm}).
\end{proof}

\begin{lemma}One has the following formula for $\hat{u}$: $\forall t\geq0$, almost surely,
\begin{equation}\label{It\^o'sdiff}
\begin{split}
&  \int_{\mathcal{O}}\left| \hat{u}_t(x) \right|^l  \, dx  +
\int_0^t \cE \, \big(l\, (\hat{u}_s)^{l -1}sgn(\hat{u}_s) , \,
\hat{u}_s \big) \, ds = \int_\cO\left|\hat{\xi}(x)\right|^ldx\\
& + l\int_{0}^{t}\int_{\mathcal{O}}sgn(\hat{u}_s)
 \left| \hat{u}_s(x) \right|^{l-1} \, \hat{f}(s,x) \, dx ds -l (l-1) \, \sum_{i=1}^{d} \,\int_{0}^{t} \int_{\mathcal{O}} \left|\hat{u}_s
(x) \right|^{l -2}
\partial_i (\hat{u}_s(x))\, \hat{g}_i(s,x)\, dx \, ds \\
& + l \, \sum_{j=1}^{\infty} \int_{0}^{t}
\int_{\mathcal{O}}sgn(\hat{u}_s)
 \,  \left|\hat{u}_s(x) \right|^{l-1}\, \hat{h}_j(s,x)\, dx dB_{s}^{j}+ \frac{l (l-1)}{2}\, \sum_{j=1}^{\infty}
\int_{0}^{t}\int_{\mathcal{O}}
 \left|\hat{u}_s(x) \right|^{l-2}\, \hat{h}_j^2(s,x,)\, dx\, ds\\&+l\int_{0}^{t}\int_{\mathcal{O}}sgn(\hat{u}_s)
 \left|\hat{u}_s(x) \right|^{l-1} (\nu^1-\nu^2)(dx\, ds)\, .\\
\end{split}
\end{equation}
\end{lemma}
\begin{proof}
From the It\^o formula for the difference of two solutions (see
Theorem 6 in \cite{DMZ12}), we have $ P$-almost surely for all $t\in
[0,T]$ and $ n \in \bbN^*$
\begin{equation*}
\begin{split}
& \int_{\mathcal{O}} \varphi_n (\hat{u}_t(x)) \, dx \; + \; \int_0^t
\mathcal{E} \big(\varphi_n^{\prime} (\hat{u}_s), \, \hat{u}_s \big)
\, ds =\int_\cO\varphi_n(\hat{\xi}(x))dx\\&+
\int_{0}^{t}\int_{\mathcal{O}} \varphi_n^{\prime}(\hat{u}_s(x))\,
\hat{f}(s,x) \, dx ds  - \sum_{i=1}^{d} \int_{0}^{t}
\int_{\mathcal{O}} \varphi_n^{\prime \prime } (\hat{u}_s(x))
\partial_i \hat{u}_s(x) \, \hat{g}_i(s,x) \, dx \, ds \\& +
\sum_{j=1}^{\infty} \int_{0}^{t} \int_{\mathcal{O}}
\varphi_n^{\prime} (\hat{u}_s(x)) \, \hat{h}_j(s,x)\, dx dB_{s}^{j}+
\frac{1}{2}\, \sum_{j=1}^{\infty} \int_{0}^{t}\int_{\mathcal{O}}
\varphi_n^{\prime \prime } (\hat{u}_s(x))\, \hat{h}_j^2(s,x) \, dx\,
ds\\&
+\int_0^t\int_\cO\varphi_n^{\prime}(\hat{u}_s(x))(\nu^1 -\nu^2)(dxds) \, .\\
\end{split}
\end{equation*}
%The last term is negative and the two terms in the LHS are positive, hence, thanks to the preceding estimates, the last term is controlled by the other terms in the RHS which are bounded.
Then, passing to the limit as $ n \to \infty $, the convergences
come from the dominated convergence theorem.
\end{proof}

Similar as before,  we define the processes $\hat{v}$ and $\hat{v}'$
by
\begin{equation*}
\begin{split}
\hat{v}_t :&=\sup _{s\leq t}\left( \int_{\cO}\left|
\hat{u}_s\right|^l dx +\gamma l\left( l-1\right)
\int_0^s\int_{\cO}\left| \hat{u}_r\right|
^{l-2}\left| \nabla \hat{u}_r\right| ^2\, dx\, dr\right)  \\
 \hat{v}_t^{\prime } :&=\int_\cO\left| \hat{\xi}\right|^ldx+l^2 c_1
\left\| \left| \hat{u}\right| ^l\right\| _{1,1;t}
+l\left\|\hat{f}^0\right\| _{\theta ,t}^{*}\left\| \left| \hat{u}
\right| ^{l-1}\right\| _{\theta ;t} \\&+ l^2 \left( c_2\left\|
|\hat{g}^0 |^2\right\| _{\theta;t}^{*}
 + c_3\left\| |\hat{h}^0 |^2\right\| _{\theta ;t}^{*}\right) \left\|
\left| \hat{u}\right|^{l-2}\right\|_{\theta ;t},\\
\end{split}
\end{equation*}
where above and below $\gamma$, $c_1$, $c_2$ and $c_3$ are the constants given by relations \eqref{const}.\\
We remark first that the last term in  (\ref{It\^o'sdiff})  is non
positive,
indeed:\begin{eqnarray*}&&\int_{0}^{t}\int_{\mathcal{O}}sgn(\hat{u}_s)
 \left| S_s-u^2_s (x) \right|^{l-1} (\nu^1-\nu^2)(dx\, ds)\\&=&\int_{0}^{t}\int_{\mathcal{O}}sgn(S_s-u^2_s)
 \left|S_s-u^2_s (x)   \right|^{l-1} \nu^1(dx\, ds)\\&&-\int_{0}^{t}\int_{\mathcal{O}}sgn(u^1_s-S_s)
 \left|u^1_s (x)-S_s(x) \right|^{l-1} \nu^2(dx\, ds)\leq0.\end{eqnarray*}
Then applying the same  proof as the one of Lemma \ref{tau}, we
obtain:\begin{equation*}
\begin{split}
& \tau E \sup _{0 \leq s\leq t}\left( \int_{\cO}\left|
\hat{u}_s\right|^l\, dx +\gamma l\left( l-1\right)
\int_0^s\int_{\cO}\left| \hat{u}_r\right| ^{l-2}\left| \nabla
\hat{u}_r\right| ^2dx dr\right) \\& \leq E\int_\cO\left|
\hat{\xi}\right|^ldx+ l^2c_1 E \left\| \left| \hat{u}\right|
^l\right\| _{1,1;t} +l E \left\|\hat{f}^0\right\| _{\theta
,t}^{*}\left\| \left| \hat{u}\right| ^{l-1}\right\| _{\theta
;t}\\&\quad+l^2 E\left( c_2 \left\| |\hat{g}^0 |^2\right\| _{\theta
;t}^{*}+c_3\left\| |\hat{h}^0 |^2 \right\| _{\theta ;t}^{*}\right)
\left\| \left| \hat{u}\right|^{l-2}\right\|_{\theta ;t}.
\end{split}
\end{equation*}
and this yields that the process $\tau \hat{v}$ is dominated by $\hat{v}'$.\\
Starting from here, we can repeat line by line the proofs of Lemmas
15-17 in \cite{DMS05} and apply the Moser iteration as at the end of
Subsection \ref{casborn} to obtain the desired estimations, namely:
\begin{lemma}{\label{Lemme63}}
There exists a function $k_2:\bbR_{+}\rightarrow \bbR_{+}$ which
involves only the structure constants of our problem and such that
the
following estimate holds%
$$
E\|\hat{u}\|^p_{\infty,\infty :t}\le k_2\left( t\right) E\left(
\left\| \hat{\xi}\right\| ^p +\left\|\hat{f}^0\right\| _{\theta
;t}^{*p}+\left\| |\hat{g}^0 |^2 \right\| _{\theta
;t}^{*\frac{p}{2}}+\left\| |\hat{h}^0 |^2 \right\| _{\theta
;t}^{*\frac{p}{2}}\right) .
$$
\end{lemma}
\begin{lemma}\label{estimvdiff}
There exists a function $k_1:\bbR_+\times\bbR_+\rightarrow\bbR_+$
which involves only the structure constants of our problem and such
that the following estimate holds
\begin{eqnarray*}
E\hat{v}_t\leq
k_1(l,t)E\left(\int_\cO|\hat{\xi}|^ldx+\left\|\hat{f}^0\right\|_{\theta;t}^{*l}+\left\|\left|\hat{g}^0\right|^2\right\|_{\theta;t}^{*\frac{l}{2}}+\left\|\left|\hat{h}^0\right|^2\right\|_{\theta;t}^{*\frac{l}{2}}\right).
\end{eqnarray*}
\end{lemma}

\begin{center}
\begin{minipage}[t]{7cm}
Laurent DENIS \\
Laboratoire d'Analyse et Probabilit\'es\\
 Universit{\'e} d'Evry Val
d'Essonne\\
23 Boulevard de France\\
 F-91037 Evry Cedex, FRANCE\\
 e-mail: ldenis{\char'100}univ-evry.fr
\end{minipage}
\hfill
\begin{minipage}[t]{7cm}
 Anis MATOUSSI \\
 LUNAM Université, Université du Maine\\
 Fédération de Recherche 2962 du CNRS\\
 Mathématiques des Pays de Loire \\
Laboratoire Manceau de Mathématiques\\
 Avenue Olivier Messiaen\\ F-72085 Le Mans Cedex 9, France \\
email : anis.matoussi@univ-lemans.fr\\
and \\
CMAP,  Ecole Polytechnique, Palaiseau
\end{minipage}
\hfill \vspace*{0.8cm}
\begin{minipage}[t]{12cm}
Jing ZHANG \\
Laboratoire d'Analyse et Probabilit\'es\\
 Universit{\'e} d'Evry Val
d'Essonne\\
23 Boulevard de France\\
 F-91037 Evry Cedex, FRANCE\\
Email: jing.zhang.etu@gmail.com

\end{minipage}
\end{center}

\end{document}